\documentclass[onecolumn,a4paper,reqno]{amsart}
\usepackage{t1enc}
\usepackage[latin2]{inputenc}
\usepackage[mathscr]{eucal}
\usepackage{amsmath,dsfont}
\usepackage{amssymb}
\usepackage{amsthm}
\usepackage{enumerate}
\usepackage{multirow}
\usepackage[all]{xy}
\newtheorem{tetel}{Theorem}[section]
\newtheorem{defi}[tetel]{Definition}
\newtheorem{cor}[tetel]{Corollary}
\newtheorem{lemma}[tetel]{Lemma}
\newtheorem{prop}[tetel]{Proposition}
\newtheorem{rem}[tetel]{Remark}
\newtheorem{notation}[tetel]{Notation}
\newcommand{\cupdot}{\mathbin{\mathaccent\cdot\cup}}
\def\Irr{\operatorname{Irr}}
\def\C{{\mathbb C}}
\begin{document}
\title{The  depth of the maximal subgroups of Ree groups}
\author{L.~H\'ethelyi}
\address{Department of Algebra, Budapest University of Technology and Economics, H-1521 Budapest, M\H uegyetem rkp. 3--9.}
\email{fobaba@t-online.hu}

\author{E.~Horv\'ath}
\address{Department of Algebra, Budapest University of Technology and Economics, H-1521 Budapest, M\H uegyetem rkp. 3--9.}
\email{he@math.bme.hu}

\author{F.~Pet\'enyi}
\address{Department of Algebra, Budapest University of Technology and Economics,
 H-1521 Budapest, M\H uegyetem kp. 3--9.}
\email{petenyi.franciska@gmail.com}
\dedicatory{To the memory of Professor K. Corr\'adi }
\abstract{%
We determine the combinatorial  and the ordinary depth of the maximal subgroups of the simple Ree
 groups, $R(q)$. 
}
}
\maketitle

\section{Intoduction}
Similarly to \cite{F}, \cite{F2}, \cite{FKR} and \cite{HHP}, we will study the  combinatorial  and ordinary depth of some subgroups of a certain group class.
We will determine the depth of maximal subgroups of Ree groups $R(q)$, $q\geq 27$.
 
 Originally depth was defined for von-Neumann algebras, see \cite{GHJ}.
Later it was also defined for Hopf algebras, see \cite{KN}. For some recent results in this direction, see
\cite{K} and \cite{KHSZ}. In \cite{KK} and later in \cite{BKK} the
 depth  of semisimple algebra inclusions was studied. The ordinary depth of a group inclusion $H\leqslant G$ (denoted by $d(H,G)$) is defined as the depth of group algebra inclusion $\C H\subseteq \C G$, studied in \cite{BDK}. It is proved in \cite{BDK} that the ordinary depth of a subgroup $H$ of a finite group  $G$ is  bounded  from above by the combinatorial depth $d_c(H,G)$.
 In particular, it is always finite. 
For the  original definitions of $d_c(H,G)$ and of $d(H,G)$, see \cite{BDK}.  Here  we will use some equivalent forms of the original definitions.\\

Let us denote by $H^x:=x^{-1}Hx$ and $H^{(x_1,\dots,x_n)}:=H\cap H^{x_1}\cap\dots \cap H^{x_n}$, for $x_1,\dots,x_n\in G$.
Let us denote by $\mathcal{U}_{\,0}(H):=H$,
 $\mathcal{U}_{\,1}(H):=\{ H\cap H^x | x\in G\}$, and in general
let $ \mathcal{U}_{\,n}(H):=\{
 H^{(x_1,\dots,x_n)}| x_1,\dots,x_n\in G\}$\\

\begin{defi}\label{jellemzes}\cite[Thm. 3.9]{BDK}
Let $H$ be a subgroup of finite group $G$. The {\em combinatorial depth} $d_c(H,G)$ is the minimal possible positive integer  which can be determined from the
following upper bounds: 
\begin{enumerate}
\item{Let $i\geq 1$. The combinatorial depth $d_c(H,G)\leq 2i$ if and only if for
 any $x_1,\dots,x_i\in G$, there exist $y_1,\dots, y_{i-1}\in G$
 with $H^{(x_1,\dots,x_i)}=H^{(y_1,\dots ,y_{i-1})}$.
(In other words, $d_c(H,G)\leq 2i$ if and only if
$\mathcal{U}_{\;i}(H)=\mathcal{U}_{\,i-1}(H)$ )}
\item{Let $i> 1$. The the combinatorial depth $d_c(H,G)\leq 2i-1$ if and only if for any $x_1,\dots,x_i\in G$, there exist $y_1,\dots, y_{i-1}\in G$ with $H^{(x_1,\dots,x_i)}=H^{(y_1,\dots ,y_{i-1})}$ and $x_1hx_1^{-1}=y_1hy_1^{-1}$ for any $h\in H^{(x_1,\dots, x_i)}$.}
\item{ The combinatorial depth $d_c(H,G)=1$ if and only if  for every $x\in G$ there exists  an element
$y\in H$  such that $xhx^{-1}=yhy^{-1}$, for every $h\in H$.(In other words:
$G=HC_G(H)$.)}
\end{enumerate}
\end{defi}

It is easy to see that $d_c(H,G)\leq 2 $ if and only if $H\triangleleft G$.
\\

We define the ordinary depth   using an equivalent form of it, more details are in \cite{BKK}.\\
Two irreducible characters $\alpha,\beta \in \Irr(H)$ are {\em related},
$\alpha \sim \beta$, if they are constituents of  $\chi|_H$, for some
 $\chi \in \Irr(G)$. The {\em distance} $d(\alpha,\beta)=m$
is the smallest integer $m$ such that there is a chain of irreducible
 characters of $H$ such that $\alpha=\psi_0 \sim \psi_1 \sim \ldots \sim
\psi_m=\beta $. If there is no such chain then
$d(\alpha,\beta)=-\infty $ and if $\alpha=\beta $ then the distance is zero.
 If $X$ is the set of irreducible constituents of $\chi|_H$ then
{\em $m(\chi):=\max_{\alpha \in \Irr(H)}\min_{\psi \in  X} d(\alpha,\psi)$.}

\begin{defi}\cite[Thm~3.9, Thm~ 3.13]{BKK}\cite{KK}\label{distance} Let $H$ be a subgroup of a finite group $G$.
The {\em ordinary depth} $d(H,G)$ is the minimal possible  positive integer  which can be determined from the following upper bounds: 
\begin{itemize}
\item[(i)] Let $m\geq 1$. The  ordinary depth $d(H,G)$ $\leq 2m+1$ 
if and only if the distance between two irreducible characters of $H$ is at most $m$.
\item[ (ii)] Let $m\geq 2$. The   ordinary depth $d(H,G)\leq 2m$ 
if and only if $m(\chi)\leq m-1$ for all $\chi \in Irr(G)$.
\item[ (iii)]   $d(H,G)\leq 2$ if and only if $H$ is normal in $G$,  $d(H,G)=1$
if and only if $G=HC_G(x)$ for all $x\in H$.
\end{itemize}
\end{defi}

\begin{tetel}\cite[Thm~6.9]{BKK}\label{inter}
Suppose that  $H$ is a subgroup of a finite group $G$ and $N=Core_G(H)$
is the intersection  of $m$ conjugates of $H$. Then $H$ has ordinary depth
$\leq 2m$ in $G$. If $N\leq Z(G)$ also holds, then $d(H,G)\leq 2m-1$.
\end{tetel}

A trivial consequence of  Theorem \ref{inter} is that if $G$ is simple and $H$ is a proper subgroup having a disjoint conjugate, 
 then  $d(H,G)=3$.\\

The paper is organized as follows. In Section 2 first we introduce the main information about Ree groups, after that in each section we determine the 
 combinatorial and ordinary depth of a fixed maximal subgroup, more precisely:
for  $N_G(P)$ ($P\in Syl_3(G)$), $N_G(M^{\pm1})$ ($M^{\pm1}\in Hall_{q\pm 3m+1}(G)$), $N_G(M)$ ($M\in Hall_{\frac{q+1}{4}(G)}$), $C_G(i)$ ($o(i)=2$), $G_0\simeq R(q_0),q_0>3$ and for $q_0=3$. \\

Our main result is the following.\\

\begin{tetel} Using the notations of  Theorem \ref{ree},
 the  combinatorial and the ordinary depths of the maximal subgroups in $G=R(q)$ for $q\geq 27$ are the following:
$d_c(N_G(P),G)=d(N_G(P),G)=5$, $d_c(N_G(M^1),G)=d_c(N_G(M^{-1}),G)=4$, $d(N_G(M^1),G)=d(N_G(M^{-1}),G)=3$, $d_c(N_G(M),G)=d_c(C_G(i),G)=6$, $d(N_G(M),G)=d(C_G(i),G)=3$, and $d_c(G_0, G)=4$, $d(G_0, G)=3$.
\end{tetel}
\section{General information on Ree groups}
First let us recall some facts about Ree groups, $ ^2G_2(q)=R(q)$.
\begin{tetel}\label{ree}\cite[Ch XI. Thm. 13.2, p.292]{Hup}, \cite{W}, \cite{LN},
\cite[L. 2,L. 3]{zll}
Suppose that $q=3^{2n+1}$ ($n\geq 1$), then there exist groups $G=R(q)$ of order $(q^3+1)q^3(q-1)$ with the following properties:
 \begin{enumerate}[(a)]
 \item{$G$ is simple.}
 \item{$G$ is doubly transitive  on $\Omega=\{1,\dots, q^3+1\}$. The stabilizer  $G_1$ of point $1$ is $N_G(P)$, where $P\in Syl_3(G)$. The stabilizer $G_{1,2}$ of two symbols is a cyclic group $W$ of order $q-1$ and $N_G(P)=PW$.
 We denote by $W_{2'}$ the $2$-prime part of $W$,
 which is a Hall subgroup of order $(q-1)/2$ and $N_G(W_{2'})=N_G(W)$ is dihedral
of order $2(q-1)$.
 If $j$ is the involution of $W$, and $P'$ is the derived subgroup of $P$,
  then $C_P(j)=C_{P'}(j)$ is  elementary abelian
of order $q$ and $C_P(j)\cap Z(P)=\{1\}$.
 If $w$ is a nontrivial element of $W$
 of odd order, then $C_P(w)=\{1\}$. The subgroup fixing exactly $3$ letters has order
 $2$.}

 \item{A Sylow $2$-subgroup $S$ of $G$ is  elementary abelian of order $8$.
 Moreover, $C_G(S)=S$ and $|N_G(S)|=8\cdot 7\cdot 3$. The $2-$subgroups of equal
order are conjugate in $G$.}
 \item{$G$ has cyclic Hall subgroups $M^i$ ($i=\pm 1$) of orders
 $q+1+i \cdot 3m$, where $m=3^n$. Each  subgroup $M^i$ is TI and it is
 the centralizer of each of its non-identity elements. The subgroups
 $B^i:=N_G(M^i)$ are Frobenius groups  with  kernel $M^{i}$ and cyclic complement of order $6$.}
 \item For each subgroup $V$ of order $4$ there  exists a cyclic Hall TI-subgroup
 $M$ of order $\frac{q+1}{4}$ and an element $t$ of order $6$ such that $N_G(V)=N_G(M)=V\rtimes (M\rtimes <t>)\simeq (V\times (M\rtimes C_2))\rtimes C_3)$,
 $C_G(M)=V\times M$ and $C_G(V)\simeq V\times (M\rtimes C_2)$.
 \item For each nontrivial subgroup $H\leqslant A$, where $A$ is one of the subgroups $W_{2'}, M, M^{\pm1}$, the containment  $N_G(H)\leqslant N_G(A) $ holds.
 \item{The centralizer in $G$ of an involution $j$ is isomorphic to
 $\langle j\rangle\times PSL_2(q)$.}
\item{A Sylow $3$-subgroup $P$ of $G$ has order $q^3$.  It is a
TI set. Its centre $Z(P)$ is an elementary abelian subgroup of order $q$,
$P$ is of class $3$, and $P'=\Phi(P)$ is an elementary abelian subgroup
  of order
 $q^2$ containing $Z(P)$.
  The elements of $P\backslash P'$ have order $9$, their cubes are forming
$Z(P)\backslash \{1\}$.\\
    $P$ is isomorphic to the group of all triples $(x,y,z)$ with $x,y,z\in GF(q)$ and the following multiplication rule:\\
$(x_1,y_1,z_1)(x_2,y_2,z_2)=(x_1+x_2, y_1+y_2+x_1(x_2\sigma), z_1+z_2-x_1y_2+y_1x_2-x_1(x_1\sigma)x_2).$\\
Here $\sigma$ is the automorphism of $GF(q)$ such that $x\sigma^2=x^3$ for all $x\in GF(q)$. We have}\\
$P'=\Phi(P)=\{(0,y,z)\mid y,z\in GF(q)\}\mbox{ and }Z(P)=\{(0,0,z)\mid z\in GF(q)\}$
\item The maximal subgroups of $G$ are  up to conjugacy the following subgroups:\\
    $N_G(P), C_G(j), N_G(M^{\pm1}), N_G(M),G_0\simeq R(q_0),$\\
    where $q_0^a=q$ for some  prime $a$. See \cite[Theorem C]{Kleidman},
 \cite{LN},\cite[Ch 4.11.3 Thm 4.3] {RW}.

\item $|G|=2^33^{3(2n+1)}{(q-1)\over 2}{(q+1)\over 4}(q+1-3^{n+1})(q+1+3^{n+1})$, where any two of intergers $2,3,(q-1)/2,(q+1)/4,(q+1-3^{n+1}),(q+1+3^{n+1})$ are relatively prime. Each cyclic subgroup of $G$ of order $\frac{q+1}{4},q+1-3^{n+1}$
and $q+1+3^{n+1}$, respectively, can be embedded into a cyclic subgroup of order
$\frac{q^3+1}{4}=\frac{q+1}{4}(q+1-3^{n+1})(q+1+3^{n+1})$. 
 \end{enumerate}
\end{tetel}
In the following we will use the notation of Theorem \ref{ree}.
 From  Theorem \ref{ree} 
 some further  properties on the structure of $G$ can be deduced.

\begin{cor}\label{Kl}\begin{enumerate}
\item  All subgroups of order $\frac{q-1}{2}$, $\frac{q+1}{4}$, $q+3^{n+1}+1$, or $q-3^{n+1}+1$ are conjugate  in $G$, respectively.
\item The subgroups of order $q-1$ are conjugate  in $G$.
\item  The involutions of $N_G(P)$ are conjugate in $N_G(P)$.
\end{enumerate}
\end{cor}

We will need the following properties on centralizers of cylic subgroups
of $G$:

\begin{lemma}\label{centralOfqpm1} 
\begin{itemize}
\item[(i)] If $C\neq \{1\}$ is a cyclic subgroup of $G$,
 whose order divides $\frac{q-1}{2}$, then
 $C_G(C)\simeq C_{q-1}$ 
\item[(ii)]
If $C\neq \{1\}$ is a cyclic subgroup of $G$,
 whose order divides  $\frac{q+1}{4}$, then
 $C_G(C)\simeq C_2^2\times C_{\frac{q+1}{4}}$.
\end{itemize}
\end{lemma}
\begin{proof}(i) Let $C$ be a nontrivial cyclic subgroup of $G$,
 whose order divides $\frac{q-1}{2}$.
 Since the subgroups of order $\frac{q-1}{2}$  are Hall-subgroups and cyclic,
 thus
by a result of Wielandt \cite{Wie},  there is a cyclic subgroup $C_{max}$ of
 order $\frac{q-1}{2}$ such that $C\leqslant C_{max}$.
 By Theorem \ref{ree}, b)  and f) $C_{q-1}\simeq C_G(C_{max})\leq C_G(C)\leq N_G(C)\leq N_G(C_{max})\simeq D_{2(q-1)}$. Thus $C_G(C)\simeq C_{q-1}$.\\
 (ii) Similarly, if $C$ is a nontrivial cyclic subgroup of $G$ of order
 dividing $\frac{q+1}{4}$, then it is contained in a cyclic subgroup
 $C_{max}$ of order  $\frac{q+1}{4}$, and by Theorem \ref{ree} e) and f),
 $V\times C_{max}\simeq C_G(C_{max})\leq C_G(C)\leq N_G(C)\leq N_G(C_{max})\simeq (V\times (C_{max}\rtimes C_2)\rtimes C_3)$. However, by Theorem \ref{ree} c) $C_G(S)=S$, hence no $2$-element in  $N_G(C_{max})\setminus V$ centralizes $C$. Since the Sylow
 $3$-subgroups of $G$ are TI, if a $3$-element $x$ would centralize $C=\langle m \rangle$, then $m$ would normalize a Sylow $3$-subgroup, contradicting  that
$|N_G(P)|=q^3(q-1)$, which is relatively prime to $\frac{q+1}{4}$. Hence
$C_G(C)\simeq C_2^2\times C_\frac{q+1}{4}$.
\end{proof}


Using the representation of $P\in Syl_3(G)$ in Theorem \ref{ree} (h), we get the following useful results.

\begin{lemma}\label{centralizerof3orderelement} If $p\in P'\setminus Z(P)$ for
 some  $P\in Syl_3(G)$, then $C_P(p)=P'$.
\end{lemma}
\begin{proof}Thus $p=(0,b,c)$ and $b\neq 0$. Then\\
$(0,b,c)(x,y,z)=(x, b+y, c+z+bx)\mbox{ and }
(x,y,z)(0,b,c)=(x, b+y, c+z-xb).$\\
Therefore, $(x,y,z)\in C_P(p)$ if and only if $(x,y,z)\in P'$.
\end{proof}

From this,  Theorem \ref{ree} and from the character table of $N_G(P)$, see below or in
\cite{Landrock}, we deduce the following:

\begin{cor}\label{centrp}
\begin{itemize}
\item [(i)] Let $z\in Z(P)$ be a nontrivial element, then $C_G(z)=P$.
\item [(ii)] Let $y\in P'\setminus Z(P)$. The $C_G(y)=P'\rtimes \langle j \rangle$
, where $j$ is an involution.
\item[(iii)] Let $x\in P\setminus P'$ then $C_G(x)=Z(P)\langle x\rangle$.
\end{itemize}
\end{cor}

\begin{lemma}\label{conj} Let $P\in Syl_3(G)$  and $(x,y,z), (a,b,c)\in P$.
 Then $(a,b,c)^{-1}=(-a,-b+a(a\sigma),-c)$ and $(x,y,z)^{(a,b,c)}=(x,y+x(a\sigma)-a(x\sigma), z-2xb+2ya-ax(x\sigma)+ax(a\sigma)$.
\end{lemma}
\begin{proof}
Check, that $(a,b,c)(a,b,c)^{-1}=(0,0,0)$:\\
$(-a,-b+a(a\sigma),-c)(a,b,c)=(0,a(a\sigma)+(-a)(a\sigma), -(-a)b+(-b+a(a\sigma))a-(-a)(-a\sigma)a)=
(0,0,a^2(a\sigma)-a^2(a\sigma)=(0,0,0)$.\\
$(x,y,z)^{(a,b,c)}=(-a,-b+a(a\sigma),-c)(x+a, y+b+x(a\sigma), z+c-xb+ya-x(x\sigma)a)=\\=(x,y+a(a\sigma)+x(a\sigma)-a((x+a)\sigma), z-xb+ya-x(x\sigma )a+a(y+b+x(a\sigma))+(-b+a(a\sigma))(x+a)-a(a\sigma)(x+a))=\\=(x,y+x(a\sigma)-a(x\sigma), z-xb+ya-x(x\sigma)a+ ay + ab+ ax (a\sigma)-bx-ba+ax (a\sigma)+a^2(a\sigma) -ax(a\sigma)+a^2(a\sigma))=(x,y+x(a\sigma)-a(x\sigma), z-2xb+2ya-ax(x\sigma) +ax (a\sigma))$
\end{proof}

\section{The depth of $N_G(P)$}
\begin{prop}\label{3jegystabilizator}  If $\alpha,\beta,\gamma\in\{1,\dots, q^3+1\}$, then $G_{\alpha,\beta,\gamma}$ is isomorphic to $C_2$ or $1$.
Moreover, for every $\alpha,\beta \in \Omega $
 there exist $\gamma, \delta\in\Omega$ such that
 $G_{\alpha,\beta,\gamma}\simeq C_2$ and $G_{\alpha,\beta,\delta}\simeq 1$.
\end{prop}
\begin{proof} Since $G$ acts on $\{1,\dots, q^3+1\}$ doubly transitively,
 it is enough to show, that there are  $\delta,\iota,\kappa,\lambda\in\Omega$
such that $G_{1,2,\delta}=\{1\}$ and $|G_{\iota,\kappa,\lambda}|=2$. We know the second statement from Theorem \ref{ree}(b).\\
If $G_{1,2,\gamma}\simeq C_2$, then there exists a symbol, $\delta$ such that
$G_{1,2,\gamma, \delta}=\{1\}$.
 Otherwise $G_{1,2,\gamma,\delta}=G_{1,2,\gamma}$ for every $\delta$ and so
 $\{1\}\neq G_{1,2,\gamma}\subseteq Stab(1,\dots,q^3+1)$,
 which is a contradiction. Let us choose a $\delta$ such that
 $G_{1,2,\gamma,\delta}=\{1\}$, therefore $G_{1,2,\delta}\neq G_{1,2,\gamma}$ and both are in $G_{1,2}\simeq C_{q-1}$.\\
 By Theorem \ref{ree}(b) we know that $G_{1,2,\delta}\simeq C_2$ or $\{1\}$.
 Since $G_{1,2}$ contains one subgroup of order $2$,
which is $G_{1,2,\gamma}$, thus $G_{1,2,\delta}$ has to be  $\{1\}$.
\end{proof}
\begin{prop}The combinatorial depth of $N_G(P)$ is $5$.
\end{prop}
\begin{proof}
 We will use  the condition (2) in Definition \ref{jellemzes}  to prove, that $d_c(N_G(P),G)\leq 5$.
By Theorem \ref{ree} (b) we have that $G_1=N_G(P)$. Let us examine the following series of subgroups:
\[G_1\geqslant G_1^{(x_1)}\geqslant G_1^{(x_1,x_2)}\geqslant G_1^{(x_1,x_2,x_3)}.\]
Now using the knowledge on the stabilizers, the group
 $G_1^{(x_1,x_2,x_3)}=G_{1, (1)x_1, (1)x_2, (1)x_3}$ could be isomorphic to $1, C_2, W$ or $G_1$.
 Let us consider the $4$ different cases. We will show, that in every case we can find elements, $y_1$ and $y_2$ such that $G_1^{(x_1,x_2,x_3)}=G_1^{(y_1,y_2)}$.
\begin{enumerate}[(A)]
\item{If $G_1^{(x_1,x_2,x_3)}=G_1$, then $G_1^{x_i}=G_1$ holds for every $i$.
 Then $y_1=x_1$ and $y_2=x_2$ is a good choice.}
\item{If $|G_1^{(x_1,x_2,x_3)}|=q-1$, then two of the three containments
 in our series cannot be strict. Then let
$y_1:=x_1$ and $y_2:=x_2$, if $G_1\ne G_1^{(x_1)}$ or $G_1^{(x_1)}\neq G_1^{(x_1,x_2)}$,
 otherwise let $y_1:=x_1$ and $y_2:=x_3$.}
\item{If $|G_1^{(x_1,x_2,x_3)}|=2$, then there is one equality in our series.
 Also there exists an index  $i>1$ such that
 $G_1^{(x_1,\dots,x_{i-1})}=G_1^{(x_1,\dots,x_{i})}$ or $G_1=G_1^{(x_1)}$.
Let $\{y_1,y_2\}:=\{x_1,x_2,x_3\}\setminus \{x_i\}$} in the first case and let
$\{y_1,y_2\}:=\{x_2,x_3\}$ in the second case.
\item{If $G_1^{(x_1,x_2,x_3)}=\{1\}$, then by Proposition
 \ref{3jegystabilizator} there exist $\alpha,\beta$ such
 that $G_{1,\alpha, \beta}=\{1\}$.
 Since $G$ is (doubly) transitive, there are elements
$y_1,y_2\in G$ such that $\alpha =(1)y_1$ and $\beta=(1)y_2$. Thus  $G_1^{(y_1,y_2)}=\{1\}$.}
\end{enumerate}
Now we have to check, that  $x_1gx_1^{-1}=y_1gy_1^{-1}$ for any $g\in G_1^{(x_1,x_2, x_3)}$.
This is automatically true except for the case (C)
when $G_1^{x_1}=G_1$.
 Let $G_1^{(x_1,x_2,x_3)}=\langle j\rangle \backsimeq C_2$ and $G_1^{x_1}=G_1$.
 If the originally chosen $y_1,y_2$ are not suitable,
 we modify them in the
 following way. Let us choose $z\in C_G(j)\setminus G_1$, which is possible.
 Let $y_1=x_1z$. Therefore $x_1$ and $y_1$ are acting in the same way on
 $G_1^{(x_1,x_2,x_3)}$. Since $y_1\notin N_G(G_1)=G_1$, thus
  $G_1\neq G_1^{y_1}$. We get, that the subgroup
 $G_1^{(y_1)}=G_{1,(1)y_1}$ is a stabilizer of two different points.
By Proposition $\ref{3jegystabilizator}$ we know that, there is a point
 $\delta$ such that $G_{1,(1)y_1, \delta}\backsimeq C_2$.
 Furthermore, $G_{1,(1)y_1, \delta}\subseteq G_{1,(1)y_1}\simeq W$
 and $G_{1,(1)y_1}$ contains only one involution, $j$.
 If we choose an element $y_2$ such that $(1)y_2=\delta$,
 then $G_1^{(y_1,y_2)}=G_{1,(1)y_1, \delta}=\langle j \rangle=G_1^{(x_1,x_2,x_3)}$, and we are done.\\
\indent
To prove that $d_c(N_G(P),G)>4$ we will use condition  (1) in Definition \ref{jellemzes}. 
 Let $\beta,\gamma$ such that $G_{1,\beta,\gamma}=1$, which is possible by
 Proposition \ref{3jegystabilizator}. Let $x_1, x_2\in G$ such that $(1)x_1=\beta$, $(1)x_2=\gamma$. Since $G$ is transitive, such $x_1$ and $x_2$ exist.
 Thus ${G_1}^{(x_1,x_2)}=G_{1,\beta,\gamma}=\{1\}$, so by Theorem \ref{ree} (b)
 there is no element $y_1$ such that ${G_1}^{(x_1,x_2)}={G_1}^{(y_1)}$.
 This implies, that $d_c(N_G(P),G)=5$.

\end{proof}
Below we present the character table of $N_G(P)$ (see \cite{Landrock}). To shorten our notation, let $\zeta:=1+i \sqrt{3}m$ and $\xi:=\frac{1}{2}(m+\sqrt{3}m\ i)$.\\
The elements $X$, $Y$ and $T$ are fixed elements in $Z(P)$,  $P\setminus P'$ and $ P'\setminus Z(P)$, respectively. The element $J$ is the involution in $W$ and the element $R$ is a generator of $W_{2'}$. Furthermore $\epsilon$ is a primitive root of unity of order $\frac{q-1}{2}$ and $a, b\in \mathbb{Z}_{\frac{q-3}{2}}\setminus\{0\}$.

\[
\begin{array}{l|cccccccccccc}
 &1 & X & Y & T & T^{-1} & YT & YT^{-1} & JT & JT^{-1} & R^a & JR^a & J \\\hline
\mathds{1} & 1 & 1 & 1 & 1 & 1 & 1 & 1 & 1 & 1 & 1 & 1 & 1 \\
\Delta & 1 & 1 & 1 & 1 & 1 & 1 & 1 & -1 & -1 & 1 & -1 & -1 \\
\psi_b^+& 1 & 1 & 1 & 1 & 1 & 1 & 1 & 1 & 1 & \epsilon^{ab} & \epsilon^{ab} & 1 \\
\psi_b^-&1 & 1 & 1 & 1 & 1 & 1 & 1 & -1 & -1 & \epsilon^{ab} & -\epsilon^{ab} & -1 \\
\alpha_1& q-1 & q-1 & -1 & q-1 & q-1 & -1 & -1 & 0 & 0 & 0 & 0 & 0 \\
\alpha_2& (q-1) q & -q & 0 & 0 & 0 & 0 & 0 & 0 & 0 & 0 & 0 & 0 \\

\alpha_3& m (q-1) & m (q-1) & -m & -m \overline{\zeta}  & -m\zeta &  \xi & \overline{\xi} & 0 & 0 & 0 & 0 & 0 \\

\alpha_4& m (q-1) & m (q-1) & -m & -m \zeta  & -m\overline{\zeta} & \overline{\xi } &  \xi & 0 & 0 & 0 & 0 & 0 \\

\alpha_5& \frac{m}{2} (q-1) & \frac{1}{2} m (q-1) & m &  - \frac{1}{2}m \overline{\zeta} & -\frac{1}{2} m \zeta & - \xi& - \overline{\xi} &- \frac{1}{2} \overline{\zeta} & - \frac{1}{2} \zeta & 0 & 0 & \frac{q-1}{2} \\

\alpha_6& \frac{m}{2} (q-1) & \frac{1}{2} m (q-1) & m &  - \frac{1}{2}m \overline{\zeta} & -\frac{1}{2} m \zeta  & - \xi & - \overline{\xi} &  \frac{1}{2} \overline{\zeta} &  \frac{1}{2} \zeta & 0 & 0 &-\frac{q-1}{2} \\

\alpha_7& \frac{m}{2} (q-1) & \frac{1}{2} m (q-1) & m & -\frac{1}{2} m \zeta & -\frac{1}{2} m \overline{\zeta } & -\overline{\xi} & -\xi & - \frac{1}{2} \zeta & - \frac{1}{2} \overline{\zeta} & 0 & 0 & \frac{q-1}{2} \\

\alpha_8& \frac{m}{2} (q-1) & \frac{1}{2} m (q-1) & m & -\frac{1}{2} m \zeta & -\frac{1}{2} m \overline{\zeta}  & - \overline{\xi} & - \xi & \frac{1}{2} \zeta &  \frac{1}{2} \overline{\zeta} & 0 & 0 & -\frac{q-1}{2}
\end{array}
\]
 The centralizers of the element can be found in \cite{Landrock}, using them we can get that $|Cl(X)|=q-1$, $|Cl(Y)|=\frac{1}{3}q^2(q-1)$, $|Cl(T)|=|Cl(T^{-1})|=\frac{1}{2}q(q-1)$, $|Cl(YT)|=|Cl(YT^{-1})|=\frac{1}{3}q^2(q-1)$, $|Cl(JT)|=|Cl(JT^{-1})|=\frac{1}{2}q^2(q-1)$, $|Cl(R^a)|=|Cl(JR^{a})|=q^3$, and $|C(J)|=q^2$.\\

To compute the induced characters $\mathds{1}^{G}$ $\Delta^G$ and ${\psi_b^{+}}^{G}$, ${\psi_b^{-}}^{G}$ we need the following lemma.
\begin{lemma}
\begin{enumerate}[a)]
\item Let $p\in P$. For an element $x\in G$, the element  $p^{x}$ is in $N_G(P)$ if and only if $x$ is in $N_G(P)$.
\item Let $i$ be an involution in $N_G(P)$. For an element $x\in G$, the element $i^x$ is in $N_G(P)$ if and only if $x$ is in $C_G(i)N_G(P)$.
\item Let $w\in W\setminus \{i\}$. For an element $x\in G$, the element $w^x$ is in $N_G(P)$ if and only if $x$ is in $N_G(W)N_G(P)$.
\item Let $i$ be an involution in $N_G(P)$ and $p\in P$ such that $o(ip)=6$. For an element $x\in G$, the element $(ip)^x$ is in $N_G(P)$ if and only if $x$ is in $N_G(P)$.
\end{enumerate}
\end{lemma}
\begin{proof}
\begin{enumerate}[a)]
\item Let assume that $p\in P, P^{x^{-1}}$. Since the Sylow $3$-subgroups in $G$ are TI, we get that $x\in N_G(P)$. The other direction is trivial.
\item Let assume that $i,i^{x}\in N_G(P)$. Since the involutions in $N_G(P)$ are conjugate, there is an element $n\in N_G(P)$ such that $(i^{x})^{n}=i$. Thus $xn\in C_G(i)$, or equivalently $x\in C_G(i) N_G(P)$. The other direction is trivial.
\item  Without loss of generality we can suppose that $w\in W_{2'}$ and $w^x\in N_G(P)$. Otherwise, we raise $w$ into a suitable $2$-power. Since $W_{2'}$ is a Hall subgroup of order $\frac{q-1}{2}$, by a result of Wielandt \cite{Wie}, we have that $W_{2'}$, $W_{2'}^x$ are conjugate in $N_G(P)$.
    Thus, $w,(w^x)^n\in W_{2'}$ for some $n\in N_G(P)$. 
   Since $W_{2'}$ is cyclic,  by Theorem \ref{ree} (f), and (b) we get that
    $xn\in N_G(\langle w\rangle)\leqslant N_G(W_{2'})=N_G(W).$
    Thus $x\in N_G(W)N_G(P)$. The other direction is trivial.
\item By the assumption,  $(ip)^2\in P$.  On the other hand, $((ip)^2)^x\in N_G(P)$, thus by using one direction of Part a), we have that $x\in N_G(P)$. The other direction is trivial.
\end{enumerate}
\end{proof}Using the previous Lemma we get the following
\begin{cor}\label{4} The characters $\mathds{1}^G\mid_{N_G(P)}$, $\Delta^G\mid_{N_G(P)}$, $(\psi_b^+)^G\mid_{N_G(P)}$ and $(\psi_b^-)^G\mid_{N_G(P)}$ have the following values:
\[\begin{array}{l|cccccccccccc}
 \mathds{1^G}|_{N_G(P)}& q^3+1 & 1 & 1 & 1 & 1 & 1 & 1 & 1 & 1 & 2 & 2 & q+1\\
 {\Delta^G}|_{N_G(P)}    &q^3+1 & 1 & 1 & 1 & 1 & 1 & 1 & -1 & -1 & 2 & -2 & -q-1\\
 {(\psi_b^{+})^G}|_{N_G(P)} &q^3+1 & 1 & 1 & 1 & 1 & 1 & 1 & 1 & 1 & 2\epsilon^{ab} & 2\epsilon^{ab} & q+1\\
 {(\psi_b^{-})^G}|_{N_G(P)} &q^3+1 & 1 & 1 & 1 & 1 & 1 & 1 & -1 &- 1 & 2\epsilon^{ab} & -2\epsilon^{ab} & -q-1
\end{array}.
\]
\end{cor}
\begin{prop} The ordinary depth of $N_G(P)$ is $5$.
\end{prop}
\begin{proof}
We determined the irreducible constituents of $\mathds{1}^G|_{N_G(P)}-2(\mathds{1}|_{N_G(P)})$, $\Delta^G|_{N_G(P)}-2\Delta$, $(\psi_b^{+})^G|_{N_G(P)}-2 \psi_b^{+}$ and $(\psi_b^{-})^G|_{N_G(P)}-2 \psi_b^{-}$. Since these characters have zero values where the different $\psi_b^+$, (or $\psi_b^{-}$) differ from each other,
 we get that the  multiplicity of $\psi^+_{b'}$  (or of $\psi^-_{b'}$) in these  four 
characters is independent of $b'$. 

With the help of this, we determined the irreducible constituents of the characters
in Corollary \ref{4}.

The computed results are the following:
\begin{align*}
&\mathds{1}^{G}|_{N_G(P)}=2\, \mathds{1}|_{N_G(P)}+\alpha_1+q\alpha_2+m\alpha_3+m\alpha_4+\frac{m+1}{2}\alpha_5+\frac{m-1}{2}\alpha_2+\frac{m+1}{2}\alpha_7+\frac{m-1}{2}\alpha_8,\\
&\Delta^{G}|_{N_G(P)}=2 \, \Delta+\alpha_1+q\alpha_2+m\alpha_3+m\alpha_4+\frac{m-1}{2}\alpha_5+\frac{m+1}{2}\alpha_6+\frac{m-1}{2}\alpha_7+\frac{m+1}{2}\alpha_8,\\
&(\psi_b^+)^{G}|_{N_G(P)}=2 \, \psi_b^++\alpha_1+q\alpha_2+m\alpha_3+m\alpha_4+\frac{m+1}{2}\alpha_5+\frac{m-1}{2}\alpha_6+\frac{m+1}{2}\alpha_7+\frac{m-1}{2}\alpha_8,\\
&(\psi_b^-)^{G}|_{N_G(P)}=2 \, \psi_b^-+\alpha_1+q\alpha_2+m\alpha_3+m\alpha_4+\frac{m-1}{2}\alpha_5+\frac{m+1}{2}\alpha_6+\frac{m-1}{2}\alpha_7+\frac{m+1}{2}\alpha_8.
\end{align*}
 The distance between $\beta,\gamma\in \Irr(N_G(P))$ is one in $G$ if and only if $[\beta^G|_{N_G(P)},\gamma]\neq 0$. Hence the distance between arbitrary elements of $\{\mathds{1},\Delta,\psi_b^+,\psi_b^-\}$ and arbitrary elements of $\{\alpha_i\}_{i=1}^8$ is one. In particular, the distance between two irreducible characters of $N_G(P)$ is at most $2$. Thus by condition (i) in
 Definition \ref{distance} we get that $d(N_G(P),G)\leq 5$. Clearly $d(\mathds{1},\Delta)=2$ Moreover,\\
$m(\mathds{1}_G)=\max_{\alpha\in \Irr(N_G(P))} \min_{\chi\in \{\mathds{1|_{N_G(P)}}\}} d(\alpha,\chi)=\max_{\alpha \in \Irr(N_G(P))}  d(\alpha,\mathds{1|_{N_G(P)}})=2.$
Thus by  condition (ii) in  Defintion \ref{distance} we get that $d(N_G(P),G)=5$.
\end{proof}

\section{The depth of $N_G(M^1)$ and $N_G(M^{-1})$}
%


\begin{prop}We have that $d_c(B^{1},G)=d_c(B^{-1},G)=4$ and $d(B^{1},G)=d(B^{-1},G)=3$
\end{prop}
\begin{proof}
We prove the statement for $B^1$, the proof for $B^{-1}$ is similar.\\
We use condition (1) in  Definition \ref{jellemzes}  to prove $d_c(B^1,G)\leq 4$.\\
If $x_1\in B^1$, then $(B^1)^{(x_1)}=B^1$. If $x_1\notin B^1$, then
$(B^1)^{(x_1)}$ is a subgroup of a Frobenius complement, since $M^{1}$ is TI
and $B^1=N_G(B^1)=N_G(M^1)$. Thus,  the subgroup $(B^1)^{(x_1, x_2)}$ can be either
 $B^1$
 or  isomorphic to $\{1\},C_2,C_3, C_6$.
\begin{enumerate}[(A)]
\item{If the intersection is $B^1$, then let $y_1:=x_1$, and we are done.}
\item{Now we examine the case, when the intersection is $\{1\}$.
 We have to show, that there is an element $y_1$ such that $(B^1)^{(y_1)}=\{1\}$.}\\
We will compute, how many  elements $y$ exist
 such that $(B^1)^{(y)}$ contains an involution or a  $3$-element.
\begin{itemize}
\item{First we will determine at most how many $y$ exist,
 such that $(B^1)^{(y)}$ contains an involution. We observe that:}\\
\textit{If the involution $i$ is an element of $B^1\cap (B^1)^{y}$, then $y\in B^1C_G(i)$.}\\
To prove this: if $i\in B^1\cap (B^1)^{y}$, then there exists an element
 $i_1\in B^{1}$ such that $i_1^y=i$.
 Since $i$, $i_1$ are contained in some conjugates of the Frobenius complement of
$B^{1}$,  there is an element
 $b\in M^1$ such that $i_1=i^{b}$.
 Thus, $by\in C_G(i)$.
\end{itemize}
Let us estimate the number of elements $y$:
$|\{y\ |\ \exists i\in (B^1)^{(y)}, o(i)=2\}|\leq \sum_{\{i\in B^1\ |o(i)=2\}
}|B^1C_G(i)|\leq|M^1||B^1C_G(i)|=(q+1+3m)^2(q-1)(q+1)q.$
\begin{itemize}
\item{Now we will determine, how many elements $y$ exist,
 such that $(B^1)^{(y)}$  contains a $3$-element. We observe that:}\\
\textit{If $T\subseteq B^1\cap (B^1)^{y}$ and $|T|=3$ ,
 then $y\in B^1N_G(T)$.}\\
To prove this let us assume that $T\subseteq B^1\cap (B^1)^{y}$.
 Then there is  a subgroup $T_1\leq B^{1}$ such that $T_1^y=T$.
Since $T_1$ and $T$ are  contained in some conjugates of
 the Frobenius complement of $B^{1}$,
 there exists an element $b\in M^{1}$ such that $T_1^b=T$. Therefore $b^{-1}y\in N_G(T)$ and $y\in B^1N_G(T)$.
\end{itemize}
Since every element of $N_G(T)$ normalizes the Sylow $3$-subgroup
  containing $T$, thus\\ $|N_G(T)|\mid |N_G(P)|=q^3(q-1)$.
  By Cor. \ref{centrp} (ii), $|C_G(T)|=2q^2$.
 Since $4\nmid q-1$, we get that $|N_G(T)|=2q^2$.
 We have that: $|\{y\ |\ 3\mid  |B^1\cap(B^1)^{y}|\}|\leq|\cup_{\{T\subseteq B^1\ |\ |T|=3\}}B^1N_G(T
)|\leq|M^1||B^1N_G(T)|=2(q+1+3m)^2q^2.$

Now  some element $y$ must exist such that $(B^1)^{(y)}=\{1\}$, since:
$|\{y\ |\ (B^1)^{(y)}=1\}|\geq |G|-|\{y\ |\ \exists i:\ i\in(B^1)^{(y)},\ o(i)=2\}|-|\{y\ |\ \exists T:\ T\ \leq (B^1)^{(y)},\ |T|=3\}|\geq $
$q^3(q^3+1)(q-1)-(q+1+3m)^2(q-1)(q+1)q-2(q+1+3m)^2q^2>0.$ 
\item{Now we can assume that the intersection is nontrivial and cyclic.}\\
Let $x_i':=\{x_1,x_2\}\backslash {x_i}$.
If $x_i\in\{x_1,x_2\}$ and it
 satisfies  $(B^1)^{(x_i)}=B^1$, then
 $(B^{1})^{(x_1,x_2)}=(B^{1})^{(x_{i}')}$ holds.
 Otherwise, both $(B^1)^{(x_1)}$ and $(B^1)^{(x_2)}$ are subgroups
 of  some Frobenius complements of $B^1$.
 Since the Frobenius complements are TI sets, the intersection of $(B^1)^{(x_1)}$ and $(B^1)^{(x_2)}$  is either trivial, or one of them contains the other,
and  the intersection is the smaller one.
 \end{enumerate}
Since in every case, (A)-(C), we can choose an element $y$ such
that $(B^{1})^{(x_1,x_2)}=(B^{1})^{(y)}$, we have that $d_c(B^{1},G)\leq 4$.

Now we will show an example, where there is no convenient $y$,
 which can act in the same way as $x_1$ on the intersection.
 Let $i$ be an involution in $B^{1}$ and let us choose an element
 $x_1\in M^1\setminus C_G(i)$.
 Let $x_2\in C_G(i)\setminus B^1$. Thus
$i\in (B^1)^{(x_1,x_2)}\neq B^1.$
Let us suppose by contradiction that an element $y$ is suitable:
 $(B^1)^{(x_1,x_2)}=(B^1)^{(y)}$ and $i^{x_1}=i^y$ hold.\\
Thus $i\in (B^1)^y$ and hence $i_1=i^{y^{-1}}\in B_1$.
  Since the product of two involutions of $B^1$ is in $M^1$,
thus
$(i_1i)^y\in (M^1)^y.$
On the other hand, $(i_1i)^y=i i^y=i i^{x_1}\in M^1$.
Since $M^{1}$ is TI and $y\not \in B^{1}=N_G(M^{1})$, thus
$(i_1i)^y=ii^y\in (M^1)^y\cap M^1=\{1\}.$
Therefore, $y\in C_G(i)$, which is a contradiction, since $x_1\notin C_G(i)$, however  $i^y=i^{x_1}$.\\
Thus $d_c(B^1,G)=4$ and similarly $d_c(B^{-1},G)=4$.

We have seen that there exist elements $x_1,x_2\in G$ such that $(B^1)^{(x_1)}=(B^{-1})^{(x_2)}=\{1\}$. Since $G$ is simple, by Theorem \ref{inter} we get that $d(B^{1},G)=d(B^{-1},G)=3$.

\end{proof}

\section{The depth of $N_G(M)$}

In this section $M$ will be a fixed cyclic subgroup of $G$
 of order $\frac{q+1}{4}$. Let  $V$ be a Klein subgroup commuting
with $M$. We know, that $C_G(M)=V\times M$ and
$N_G(M)=N_G(V)\simeq V\rtimes(M\rtimes C_6)\simeq (V\times (M\rtimes C_2))\rtimes C_3$, and $C_G(V)\simeq V\times (M\rtimes C_2)$

\begin{lemma}\label{nincsq+1/4es} If there is a nontrivial element $m\in N_G(M)^{(x_1)}$, whose order divides $\frac{q+1}{4}$, then $x_1\in N_G(M)$ i.e. $N_G(M)^{(x_1)}=N_G(M)$.
\end{lemma}
\begin{proof} We know that $N_G(M)\simeq(C_2^2\times (M\rtimes C_2))\rtimes C_3$. This means, that $m$ has to be in $M\cap M^{x_1}$. Therefore, by Lemma \ref{centralOfqpm1}, we have that 
 $M,M^{x_1}\leqslant C_G(m)\simeq C_{\frac{q+1}{4}}\times C_2^2$  and
hence 
$M=M^{x_1}.$
\end{proof}
\begin{prop}
\label{NGMu1e}  Let $V$ be as above.  Then
 $\mathcal{U}_{\,1}(N_G(M))=\{N_G(M)=N_G(V), S, H\ |\ V\leqslant S\in Syl_2(G),\ [H,V]=1,\ V\neq H\simeq C_2^2 \}\cup U$, where
 $U\subseteq\{H| H\simeq \{1\}\ ,\ H\simeq C_2, \ C_3 \mbox{ or }C_6\}$.
\end{prop}
\begin{proof}
We know from Theorem \ref{ree} (c) that the Klein subgroups of $G$ are conjugate.

 By Lemma \ref{nincsq+1/4es} we have that 
 $N_G(M)^{(x_1)}$ is isomorphic to one of the following subgroups:\\
 $ 1, C_2, C_3,C_2^2, C_6,C_2^3, (C_2^2)\rtimes C_3, (C_2^3)\rtimes C_3, N_G(M).$\\
It is obvious that $N_G(M)$ occurs.
  Let us examine the other cases.
\begin{itemize}
\item[(A) $C_2^3\lesssim N_G(M)^{(x_1)}$]: \textbf{We will prove, that if
 $ N_G(M)^{(x_1)}$ is a proper subgroup of $N_G(M)$ and contains a subgroup
 isomorphic to $C_2^3$, then $V$ is contained in $N_G(M)^{(x_1)}$ and
$N_G(M)^{(x_1)}$ is isomorphic to $C_2^3$.}
Let $S$ be a subgroup of $N_G(M)^{(x_1)}$ of order $8$ and assume that
$x_1\notin N_G(M)$.
 Then $S\in Syl_2(G)$ and  since $N_G(M)=N_G(V)$, we have that $V,V^{x_1}\leqslant S$.
 Thus $V$ and $V^{x_1}$ contain exactly one common involution,
 which we denote by $i$. Suppose, that $h\in N_G(M)^{(x_1)}$ is
 an element of order $3$. Then obviously $h$ acts on $V$ and $V^{x_1}$ nontrivially, so the action is transitive on the involutions.
 This is a contradiction, since $i$ is the unique common involution of $V$ and $V^{x_1}$.  Thus, by Lemma \ref{nincsq+1/4es}, $N_G(M)^{(x_1)}$ can contain only $2$-elements, hence $S=N_G(M)^{(x_1)}$.
\\
\textit{Construction:} Let $S\in Syl_2(G)$  containing $V$.
We want to show that there exists an element
$x_1\notin N_G(M)=N_G(V)$ such that $N_G(M)^{(x_1)}=S$. Choose  an element
$x_1\in G$ such that $V\neq V^{x_1}\leqslant S$. Then $S\leqslant N_G(V)^{(x_1)}=N_G(M)^{(x_1)}$
and we are done due to the previous statement.

\item[(B) $C_2^2\lesssim N_G(M)^{(x_1)}$]: \textbf{We will prove,
that if $ N_G(M)^{(x_1)}\ne N_G(M)$  and it contains a Klein subgroup $H$,
but it does not
  contain a Sylow $2$-subgroup of $G$, then $N_G(M)^{(x_1)}=H$ and
$[H,V]=1$, moreover $V\ne H$.
 }\\
    By the assumption and Lemma \ref{nincsq+1/4es}
$N_G(M)^{(x_1)}$ can  only  be isomorphic to $C_2^2$ and $C_2^2\rtimes C_3$.
    Suppose now that $N_G(M)^{(x_1)}=H\rtimes \langle h\rangle$, where $H\simeq C_2^2$
 and $o(h)=3$. Thus $h$ acts on $H, V, V^{x_1}$ nontrivially.\\
    Since $H\leqslant N_G(V)=N_G(M)$, $H=V$ or $\langle H,V\rangle\in Syl_2(G)$.
Thus $H=V$, because otherwise $h$  could not act nontrivially on  both
 $H$ and $V$. Furthermore, $H=V^{x_1}$ due to the same explanation,
  which is a contradiction, since $x_1\not\in N_G(V)$. Thus $N_G(M)^{(x_1)}\not\simeq C_2^2\rtimes C_3$.\\
    Now we prove  that $N_G(M)^{(x_1)}$ cannot be $V$.
 If  $V\leqslant N_G(M)^{(x_1)}=N_G(V)^{(x_1)}$, then
 $C_2^3\simeq\langle V,V^{x_1}\rangle\leqslant N_G(V)^{(x_1)}=N_G(M)^{(x_1)}$, which is a contradiction.\\
\textit{Construction:}
    Let $H$ be a Klein subgroup of $G$
 different from $V$,  with $[H,V]=1$. Then $\langle H,V\rangle\in Syl_2(G)$.
 Let $H\cap V=\langle i\rangle$, and let $j$ be another generator of $V$.
 Then $ V\leqslant C_G(H)= H\times D$, where $D\simeq D_{\frac{q+1}{2}}$.
We may assume that $j\in D$, otherwise we  choose  the complement
 $\langle C_{\frac{q+1}{4}},j\rangle$ instead of $D$.
 Let $k\in D$ be an involution different from $j$.
Choose  an element $x_1\in G$ such that $V^{x_1}=\langle i,k \rangle$.
 Then $[V,V^{x_1}]\ne 1$.  The subgroup $N_G(V)^{(x_1)}=N_G(M)^{(x_1)}$ cannot contain a
Sylow $2$-subgroup  $S$ of $G$, otherwise $V,V^{x_1}\leq S$.  However,
 it contains $H$. Hence $N_G(M)^{(x_1)}=H$.
\item[(C) $1, C_2, C_3, C_6$:] $N_G(M)^{(x_1)}$ can be isomorphic to $1$, $C_2$, $C_3$ or $C_6$. To compute the depth of $N_G(V)$, however, we do not need to determine exactly which subgroups can occur.
\end{itemize}
\end{proof}
\begin{lemma} $\mathcal{U}_{\;2}(N_G(M))=\{1,C,K,S,N_G(M)=N_G(V)|\ C\simeq C_2, K\simeq C_2^2, S\simeq C_2^3 $ and each of them commutes with $ $V$ \}\cup \{X\cap Y | X,Y\in U\}$, where $U$ is as before. Thus $\mathcal{U}_{\,1}(N_G(M))\ne \mathcal{U}_{\,2}(N_G(M))$, since $V$ is in the difference set. Hence we have that
$d_c(N_G(M),G)\geq 5$.

\end{lemma}
\begin{proof}
If $Z\in \mathcal{U}_{\,2}(N_G(M))$,
then $Z=X\cap Y$ where $X,Y\in \mathcal{U}_{\,1}(N_G(M))$.
 Therefore we will study the intersection of subgroups of different
 structure from $\mathcal{U}_1(N_G(M))$ in a $\mathcal{U}_1-\mathcal{U}_1$ table.
 The elements of first the row will be the possible structure
 of $X$ and similarly the elements of the first column will be the possible structure of $Y$. These can be $N_G(M)$;  a Sylow $2$-subgroup $S$ containing $V$;  a Klein subgroup $H$, which  commutes with $V$, but not equal to it; some subgroups isomorphic to $C_6$, $C_3$, $C_2$ or $1$ (if they occur at all, we denote them by $Z_6,Z_3,Z_2,Z_1$). The element in the position $(X,Y)$ in the table will show, which structure can occur as $X\cap Y$.\\
\begin{tabular}{|c|ccccccc|}
  \hline
   & $N_G(M)$ & $S$ & $ {H}$ & $Z_6$ & $Z_3$ & $Z_2$ & $Z_1$ \\\hline

$N_G(M)$ & $N_G(M)$ & $S$ & $ {H}$ &  $Z_6$ & $Z_3$ & $Z_2$ & $1$\\

$S$   &  $S$ & $S$, ${  V}$ & $ {H}$, ${ C_2}$ & ${ C_2}$, ${ 1}$  &${ 1}$ &  $Z_2$, $1$ & $1$ \\

$ {H}$ & $ {H}$ & $ {H}$, ${ C_2}$ & $ {H}$,${ C_2}$, ${ 1}$ & ${ C_2}$, ${ 1}$  &${ 1}$ &  $Z_2$, ${ 1}$ & $1$ \\

$Z_6$   & $Z_6$ & ${ C_2}$, ${ 1}$ & ${ C_2}$, ${ 1}$ & \multicolumn{4}{c|}{\multirow{4}{*}{$X\cap Y$, where $X,Y\in U$}}   \\
$Z_3$   & $Z_3$ & ${ 1}$ & ${ 1}$ &\multicolumn{4}{c|}{}  \\
$Z_2$   & $Z_2$ & $Z_2$, ${ 1}$ & $Z_2$,${ 1}$ & \multicolumn{4}{c|}{}  \\
$Z_1$     & $1$ & $1$ &$1$  &\multicolumn{4}{c|}{}  \\
  \hline
\end{tabular}\\
It is obvious that $V$ occurs as the intersection of two Sylow $2$-subgroups containing it. Thus $V\in \mathcal{U}_{\,2}(N_G(M))$. The other Klein subgroups, 
which commute with $V$, occur already in $\mathcal{U}_{\,1}(N_G(M))$.\\
To finish the proof, we have to show, that every
  cyclic subgroup of order $2$,  which commutes with $V$, and
also $\{1\}$  occur in $\mathcal{U}_{\,2}(N_G(M))$.\\

A subgroup of order $2$ in $V$ occurs in $\mathcal{U}_{\,2}(N_G(M))$, since we can choose two different Klein subgroups   which are not equal to $V$, and commuting with it,   such that their intersection is this subgroup.
Let $\langle k\rangle$ be a subgroup of order $2$,
 which commutes with $V$ but not contained in it.
 Let $V=\langle i,j\rangle$. Then $K_1=\langle i,k\rangle$ and $K_2=\langle j,k \rangle$ both are in $\mathcal{U}_{\,1}(N_G(M))$. Thus their intersection
$\langle k\rangle$ is in $\mathcal{U}_{\,2}(N_G(M))$.\\
To see that $\{1\}\in \mathcal{U}_{\,2}$, we take two different Sylow $2$-subgroups $S_1,S_2$
containing $V$.
Let $V=\langle i,j \rangle$, and let $K_1\leq S_1 $ be of order $4$ with
 $K_1\cap V=\langle i \rangle$.
Let $K_2\leq S_2 $ be of order $4$ with $K_2\cap V=\langle j \rangle$.
Then  $K_1,K_2\in \mathcal{U}_1(N_G(M))$, thus $\{1\}=K_1\cap K_2 \in
\mathcal{U}_{\,2}(N_G(M))$.

\end{proof}We  remark, that every cyclic subgroup of order $6$, which is contained in
 $\mathcal{U}_{\,2}(N_G(M))$, is already  in $\mathcal{U}_{\,1}(N_G(M))$.

\begin{lemma}$\mathcal{U}_{\,2}(N_G(M))=\mathcal{U}_{\,3}(N_G(M))$. Thus, $d_c(N_G(M),G)\leq 6$.
\end{lemma}
\begin{proof} Similarly to the previous case, if $Z$ in $\mathcal{U}_{\,3}(N_G(M))$, then $Z=X\cap Y$, where $X\in\mathcal{U}_{\,2}(N_G(M)),\\
Y\in\mathcal{U}_{\,1}(N_G(M))$. We demonstrate the possible intersections in a  $\mathcal{U}_1-\mathcal{U}_2$   table. The first column shows the possible values $Y$ of $\mathcal{U}_1(N_G(M))$, as above. The first row  shows the possible values $X$ of $\mathcal{U}_{\,2}(N_G(M))$. They can be $N_G(M)$, the $2$-subgroups, which commute with $V$: $C_2^3$, $C_2^2$, $C_2$; $1$; and maybe some subgroups, which are isomorphic to $C_6$ or $C_3$,  denoted by $\zeta_6$ and $\zeta_3$.
  Note that  the $2$-subgroup in $Z_6$ commutes with $V$. We have seen above that   $\zeta_6=Z_6$.\\
\begin{tabular}{|c|ccccccc|}
  \hline
   & $N_G(M)$ & $C_2^3$ & $C_2^2$ & $C_2$ & $1$ & $Z_6$ & $\zeta_3$  \\\hline

$N_G(M)$ & $N_G(M)$ & $C_2^3$ & $C_2^2$ &  $C_2$ & $1$ & $Z_6$ & $\zeta_3$ \\
$S$   &  $C_2^3$ & $C_2^3$, $C_2^2$ & $C_2^2$, $C_2$ & $C_2$, $1$  &$1$ &  $C_2$, $1$ & $1$  \\

$ {H}$ & $C_2^2$ & $C_2^2$, $C_2$ & $C_2^2$,$C_2$, $1$ & $C_2$, $1$  &$1$ &  $C_2$, $1$ & $1$ \\

$Z_6$   & $Z_6$ & $C_2$, $1$ & $C_2$, $1$ &
 $C_2$, $1$ & $1$ &$Z_6$, $\zeta_3$,$C_2$, $1$ & $\zeta_3$, $1$ 
 \\
$Z_3$   & $Z_3$ & $1$ & $1$ & $1$ &$1$ &$Z_3$, $1$&$Z_3$, $1$ \\
$Z_2$   & $C_2$ & $C_2$, $1$ & $C_2$,$1$ & $C_2$,$1$ & $1$ &$C_2$,$1$ & $1$ \\
$Z_1$     & $1$ & $1$ &$1$  & $1$& $1$ & $1$ &$1$ \\
  \hline
\end{tabular}\\
The only new intersection could be isomorphic to $C_3$, if $X$ and $Y$ isomorphic to  cyclic subgroups of order $6$. In this case both $X$ and $Y$ are in $\mathcal{U}_1(N_G(M))$ and hence the intersection is  in $\mathcal{U}_2(N_G(M))$.
\end{proof}
\begin{tetel}$d_c(N_G(M), G)=6$
\end{tetel}
\begin{proof}
 We  have seen that   $5\leq d_c(N_G(M),G)\leq 6$. 
 We will give elements $x_1,x_2,x_3$ in $G$, such that one cannot find 
  any  $y_1,y_2\in G$ such that $N_G(M)^{(x_1,x_2,x_3)}=N_G(M)^{(y_1,y_2)}$
 and  the elements $x_1,y_1$ act  on $N_G(M)^{(x_1,x_2,x_3)}$ in the same way.
 We know that
$C_G(V)= V\times (M\rtimes \langle t \rangle )$,  where  $M=\langle a\rangle $, $ o(a)=\frac{q+1}{4}$ ond $o(t)=2.$
Let $x_1\in N_G(M)$ and let  $x_2$, $x_3$ be such that $N_G(M)^{(x_2)}=V\times \langle t\rangle $ and $N_G(M)^{(x_3)}=V\times \langle ta \rangle $, which is  possible by Propositon \ref{NGMu1e}. Hence we get that\\
$N_G(M)^{(x_1,x_2,x_3)}=N_G(M)\cap (V\times \langle t \rangle)\cap (V\times \langle ta \rangle)=V.$\\
Let us choose elements $y_1,y_2$ such that $N_G(M)^{(y_1,y_2)}=V$.
 Since $x_1$ normalizes $V$, if $y_1$ acts in  the same way as $x_1$,  then
$y_1$  also
normalizes $V$. Therefore,
 $V=N_G(M)\cap N_G(M)^{y_1}\cap N_G(M)^{y_2}=N_G(M)^{(y_2)}$, which
 is a contradiction  by Propositon \ref{NGMu1e}. Thus $d_c(N_G(M),G)=6$.
\end{proof}
\begin{prop}
There is an element $x\in G$ such that $N_G(M)^{(x)}=\{1\}$.
\end{prop}
\begin{proof}
We need  to estimate  the number of elements $x$, such that
 $N_G(M)^x$ contains special elements: elements of $M$,  $3$-elements of $N_G(M)$,  involutions of $N_G(M)$, respectively.
 First we make the following observations:

\begin{itemize}
\item{
We have seen in Lemma \ref{nincsq+1/4es} that if $N_G(M)^x$ contains some nontrivial elements of $M$ for some $x\in G$, then $x\in N_G(M)$. Thus we get that}
$A_{\frac{q+1}{4}}:=\{x\in G\ |\ N_G(M)^{(x)}\cap M\neq1\}=N_G(M).$
\item{
Let $p$ be a  nontrivial $3$-element in $N_G(M)$. First we show that for an element $x\in G$ the subgroup $N_G(M)^x$ contains $\langle p\rangle$ if and only if $x\in N_G(M)N_G(\langle p\rangle)$.}\\ Let $\langle p \rangle\leq N_G(M)^x$.
Since  the subgroups of order $3$ are conjugate in $N_G(M)$, there is an element $n\in N_G(M)$ such that $\langle p\rangle^{x^{-1}}=\langle p\rangle^n$.
Hence $nx\in N_G(\langle p\rangle )$, and $x\in N_G(M)N_G(\langle p\rangle)$. The other direction is trivial.
Using that $N_G(\langle p \rangle)=C_G(\langle p\rangle)$ is of order $2q^2$ by Cor.\ref{centrp} (ii), we  have: 
    $|A_3|:=|\{x\in G\ |\ N_G(M)^{(x)} \mbox{ contains $3$-elements}\}|= |\cup_{\{\langle p\rangle \leq N_G(M)| o(p)=3 \}} N_G(M) N_G(\langle p\rangle)|$ $\leq \sum_{\{\langle p\rangle \leq N_G(M)| o(p)=3 \}}\frac{6(q+1)2q^2}{6}=2q^2(q+1)^2.$
\item{There are $|C_G(i)|=q(q-1)(q+1)$ elements in $G$, which take a fixed involution to a fixed involution.
 Furthermore, the subgroup $N_G(M)=(V\times (M\rtimes C_2))\rtimes C_3$ contains $q+4$ involutions. Thus we have:}
    $|A_2|:=|\{x\in G\ |\ N_G(M)^{(x)} \mbox{ contains involutions}\}|\leq (q+4)^2q(q-1)(q+1).$
\end{itemize}
 Using all the above  results we have:
$|\{x\in G\ N_G(M)^{(x)}=1\}|\geq |G| -|A_{\frac{q+1}{4}}|-|A_3|-|A_2|=$
$q^3(q^3+1)(q-1)-6(q+1)-2q^2(q+1)^2-(q+4)^2q(q-1)(q+1).$
Since this is bigger than $1$, we get that there is an element $x\in G$ such that $N_G(M)^{(x)}=\{1\}$.
\end{proof}
Using Theorem \ref{inter} we have the following.
\begin{cor}The ordinary depth of $N_G(M)$ is $3$.
\end{cor}
\section{The depth of $C_G(i)$}

Since $C_G(i)=\langle i\rangle \times L$, where $L\simeq PSL(2,q)$, the following theorem will be useful.
\begin{tetel}(Dickson)\cite[8.27 Hauptsatz, Ch II, p. 213]{Hup1}\label{reszcsoplista}
 The complete list of all subgroups of $L=PSL(2,q)$ for $q=3^{2n+1}$ is the following:
\begin{enumerate}[a)]
\item{alternating groups $A_4$;}
\item{elementary-abelian $3$-groups of order at most $q$;}
\item{semidirect products $C_3^m\rtimes C_t$ of elementary-abelian    groups of order at most $q$ with cyclic groups of order $t$, where $t$ divides $3^m-1$ as  well as $\frac{q-1}{2}$;}
\item{groups $PSL(2, 3^m)$, if $m\mid 2n+1$;}
\item{cyclic groups of order $z$, where $z$ divides $\frac{q-1}{2}$;}
\item{dihedral subgroups of order $2z$ with $z$ as previous;}
\item{cyclic groups of order $z$, where $z$ divides $\frac{q+1}{2}$;}
\item{dihedral subgroups of order $2z$ with $z$ as previous.}
\end{enumerate}
\end{tetel}
\begin{prop}\label{U1}$\mathcal{U}_{\;1}(C_G(i))=\{C_G(i), C_P(i), C_G(V), K, C, 1\mid i\in N_G(P), P\in Syl_3(G), i\in V\simeq C_2^2, K\simeq C_2^2, [i,K]=1,i\notin K,C\simeq C_2, [i,C]=1, i\notin C\}$
\end{prop}
\begin{proof}
We will examine  what kind of  subgroups  $H\in \mathcal{U}_{\;1}(C_G(i))$ occur. If we
find a subgroup type which could occur, we will also construct it.
 For this, we choose an involution $j$ such that $C_G(i)\cap C_G(j)=H$.
 Since the involutions  are conjugate in $G$, we will be done.\\
 If $[i,i^{x_1}]=1$, then
$C_G(i)^{(x_1)}=C_G(i)\cap C_G(i^{x_1})=C_G(V), \mbox{ where }V=\langle i,i^{x_1}\rangle \mbox{ is a Klein subgroup of $G$.}$
By Theorem \ref{ree} (e), we have that
$C_G(i)\cap C_G(i^{x_1})\simeq(C_2^2\times D_{\frac{q+1}{2}}).$
We remark, that $i\in V= Z(C_G(i)^{(x_1)})$.
\begin{itemize}
\item[ $C_G(V):$] Now we show that every subgroup $H\simeq C_2^2\times D_{\frac{q+1}{2}}$, where $i\in Z(H)$, occurs in $\mathcal{U}_1{(C_G(i))}$.
 Let $j$ be another involution in $Z(H)$. Thus $<i,j>=Z(H)$ and
$C_G(Z(H))=C_G(i)\cap C_G(j)=H$ by Theorem \ref{ree} (e).
\end{itemize}
If $[i,i^{x_1}]\neq1$, then  let $D:=\langle i, i^{x_1}\rangle$.
Since  $C_G(i)= \langle i \rangle\times L$, where  $L\simeq PSL(2,q)$,
therefore the projection to the second component $\pi_2(C_G(i)^{(x_1)})=\pi_2(C_G(D))$ is a subgroup of $L\simeq PSL_2(q)$. Now we examine,
which subgroups of $PSL_2(q)$ can occur.
 We study the  subgroups of the list in Theorem \ref{reszcsoplista}
if they could be equal to $\pi_2(C_G(D))$.
\begin{itemize}
\item[a)-b)] \textbf{Let $[i,i^{x_1}]\ne 1$,  and let
$D=\langle i,i^{x_1}\rangle $ and  let us suppose that
 $\pi_2(C_G(D))$ contains an
element $p$ of order $3$. Then $C_G(D)$ also contains an element of order $3$,
moreover,
  $C_G(i)^{(x_1)}=C_P(i)\simeq C_3^{2n+1}$ for
 some $P\in Syl_3(G)$  with $i\in N_G(P)$. On the other hand,
 every subgroup of the form $C_P(i)$
occurs as $C_G(i)^{(x_1)}$ if $i\in N_G(P)$ for some $P\in Syl_3(G)$}.\\
 By raising to the second power if necessary, we may assume that
$p\in C_G(D)$.
 Let  $p\in P\in Syl_3(G)$. We remark that, by Cor. \ref{centrp},
$p\in P'\backslash Z(P)$. 
 Since the Sylow $3$-subgroups are TI,
    $i,i^{x_1}\in N_G(P).$
     In particular, $i,i^{x_1}\in P\rtimes \langle i\rangle$.
 Thus $i^{x_1}=p'i$ for a suitable $p'\in P$.
 Since $[p,i]=1$ and $[p,p'i]=1$,  thus $[p,p']=1$ also holds.
 Hence $p'\in P'$ by Lemma \ref{centralizerof3orderelement}.
If $p'\in Z(P)$, then by  Cor.\ref{centrp} we have  that $C_G(p')=P$.
 Thus $C_G(D)=C_G(i)\cap C_G(p')=C_P(i)$.
\\If $p'\notin Z(P)$, then 
 by Cor. \ref{centrp},  $C_G(p')=P'\rtimes \langle j \rangle$, where
$j$ is an involution in $N_G(P)$. Since $[p',i]\ne 1$,  we have that $j\ne i$.
Hence $j=p_1i$ for some $1\ne p_1\in P$. Thus $[j,i]\ne 1$.
Hence $C_G(D)=C_G(i)\cap C_G(p')=C_{P'}(i)=C_P(i)$, by Theorem \ref{ree} (b).

\item[$C_P(i):$] Now we will show, that every  subgroup $C_P(i)$ occurs as $C_G(D)$, if
 $P\in Syl_3(G)$  and $i\in N_G(P)$.
 Let $P\in Syl_3(G)$ such that $i\in N_G(P)$.
 The involution $i$ acts on $P'\simeq C_3^{4n+2}$ nontrivially.
 Choose an element $p\in P'$, which is not centralized by $i$.
Let $p':=[p,i]$. Then $p'\in P'$ and it is inverted by $i$.
Moreover,  $j:=p'i$ is
 an involution
such that
    $C_G(\langle i,j\rangle )=C_G(\langle i,p'\rangle)=C_G(i)\cap C_G(p')=C_P(i)
\simeq C_3^{2n+1},$ as above.
\item[c)-f)] \textbf{  Let  $[i,i^{x_1}]\ne 1$  and let
 $D:=\langle i,i^{x_1}\rangle$. Then $\pi_2(C_G(i)^{(x_1)})=\pi_2(C_G(D))$
 does not contain a nontrivial cyclic subgroup $C$, whose order
 divides $\frac{q-1}{2}$.}
Raising to the second power if necessary, we may suppose that $C_G(D)\geqslant C$.
 Then $D\leqslant C_G(C)\simeq C_{q-1}$ by Lemma \ref{centralOfqpm1},
 which is a contradiction.
\item[g)-h)] \textbf{Let
$[i,i^{x_1}]\ne 1$, and let $D=\langle i,i^{x_1}\rangle$.
 Then if $\pi_2(C_G(i)^{(x_1)})=\pi_2(C_G(D))$ is a dihedral subgroup of
order $2z$ or a cyclic subgroup of order $z$,
 where $z\mid \frac{q+1}{2}$, then $z=2$.
 Only $C_2$ and $C_2^2$ can occur as $C_G(D)$, ($C_2^3$ cannot occur).
 Obviously  $i\notin C_G(D)$ and $[C_G(D),i]=1$.
    }\\
    If $z>2$, then $\pi_2(C_G(D))$ contains a cyclic subgroup $C\neq \{1\}$
 whose size divides $\frac{q+1}{4}$. Raising to the second power the elements of $C$, if necessary, we may assume that $C\leqslant C_G(D)$.
 Thus,  by Lemma \ref{centralOfqpm1}, $D\leqslant C_G(C)\simeq C_2^2\times C_{\frac{q+1}{4}}$,
 which is a contradiction.
 Furthermore,
if $S\in Syl_2(G)$ and $C_G(D)\geqslant S$ then $i,i^{x_1}\in S$, which is a contradiction.\\

 We will show, that both $C_2$ and  $C_2^2$ occur as $C_G(D)$ if $i\notin C_2^2$ and   $[C_2^2,i]=1$, or $i\notin C_2$ and $[C_2,i]=1$.

\item[$C_2:$] Let $H=\langle k\rangle \simeq C_2$ such that $k\neq i$ and
  $[k,i]=1$. Then $i\in C_G(k)=\langle k\rangle \times U$, where $U\simeq PSL(2,q)$.
 Since  the involutions in $PSL(2,q)$ are conjugate, see 
\cite[8.5 Satz, Ch. II,  p. 193]{Hup1}, if $i\in U$, then  there exists a subgroup
 $T(\simeq D_{q-1})\leqslant U$ such that $i\in T$. Let $C\leq T$ be the unique
cyclic subgroup of order $\frac{q-1}{2}$ and
    let $j$ be  another involution in $T$.
 Then, by Lemma \ref{centralOfqpm1}, $C_G(ij)\simeq C_{q-1}$, thus 
$C_G(ij)=\langle k \rangle\times C$. Thus,
   $C_G(i)\cap C_G(j)= C_G(i)\cap C_G(ij)=\langle k\rangle $.
If $i\notin U$, then $i=ki_2$, and $i_2$ is in a dihedral subgroup $T\leq U$ of order
$q-1$.
 Let $C\leq T$ be the unique cyclic subgroup of order $\frac{q-1}{2}$  and
let $T_1:=\langle C,i\rangle$. Then  $T_1$ is also dihedral of order $q-1$. Let
$j$ be another involution in $T_1$. Then $C_G(ij)=\langle k\rangle\times C$.
As above, we have that $C_G(i)\cap C_G(j)=\langle k\rangle$.

\item[$C_2^2:$]Let $H\simeq C_2^2$ such that $i\notin H$,  and $[i,H]=1$.
 Then, by Theorem
\ref{ree} (e), $i\in C_G(H)= H\times T\simeq C_2^2 \times D_{\frac{q+1}{2}}$.
Let $C\leq T$ be the unique cyclic subgroup of order $\frac{q+1}{4}$. 
 If  $i\in T$ then  let $j$ be another   involution
in $T$. By Lemma \ref{centralOfqpm1},  $C_G(ij)\simeq C_2^2\times C_{\frac{q+1}{4}}$, thus  $C_G(ij)= H\times C $. 
   Hence, 
     $C_G(i)\cap C_G(j)= C_G(i)\cap C_G(ij)=H.$
If $i\notin T$, then $i=i_1i_2$, where $i_1\in H$ and $i_2\in T$.
Then let $C$ be as above and
let $T_1:=\langle i,C\rangle$. Then $T_1$ is also dihedral of order
 $\frac{q+1}{2}$. Let $j$ be another involution in $T_1$.
Then $C_G(ij)=H\times C$.
And hence $ C_G(i)\cap C_G(j)= C_G(i)\cap C_G(ij)=H.$
\end{itemize}
Finally we show that the trivial subgroup also occurs as intersection.
\begin{itemize}
\item[$1:$]There is an involution $j$ such that $C_G(i)\cap C_G(j)=\{1\}$\\
   Let $M^1$  be as in Theorem \ref{ree} (d).
 Let $M^3$  be a conjugate of $M^1$, which satisfies, that
 $i\in N_G(M^3)$ and let  $j$  be another involution in $N_G(M^3)$.
Since $N_G(M^3)$ is a Frobenius group with kernel $M^3$, we have that
  $\langle i,j\rangle$ is a dihedral subgroup
 of $N_G(M^3)$ and $ij$ is an element of $M^3$.
 Using 
Theorem \ref{ree} (d) and the fact,
that  $2|PSL(2,q)|$ and $q+3m+1$ are relatively prime, we get that
    $C_G(i)\cap C_G(j)=C_G(i)\cap C_G(ij)\simeq(\langle i\rangle \times
PSL(2,q))\cap C_{q+3m+1}=\{1\}$.
\end{itemize}
Hence $C_G(i)^{(x_1)}$ can have  exactly the values  mentioned in the Proposition.
\end{proof}
\begin{prop}\label{u2cgi}$
\mathcal{U}_{\;2}(C_G(i))=\mathcal{U}_{\;1}(C_G(i))\cupdot \{\langle i\rangle\}\cupdot \{V\ \mid \ i\in V\simeq C_2^2\}\cupdot \{S\ \mid \ i\in S\simeq C_2^3\}$
\end{prop}
\begin{proof} We have seen in  Theorem \ref{U1}, that $\mathcal{U}_{\;1}(C_G(i))=\{H_2\mid\ H_2\simeq C_2, i\notin H_2,\ [i,H_2]=1 \}\cupdot
\{H_4\mid H_4\simeq C_2^2, i\notin H_4,\ [i,H_4]=1 \}
\cupdot\{ H_{2q+2}\mid\ H_{2q+2}\simeq C_2^2\times D_{\frac{q+1}{2}}, i\in Z(H_{2q+2})\}\cupdot \{1, C_G(i)\}\cupdot\{ C_P(i) | P\in Syl_3(G), i\in N_G(P)\}$.
We will examine $C_G(i)^{(x_1,x_2)}$ as the intersection of
 $C_G(i)^{(x_1)}$ and  $C_G(i)^{(x_2)}$ similarly to the previous proof.
 Let us see the  $\mathcal{U}_1-\mathcal{U}_1$ table for $C_G(i)$.\\
\begin{tabular}{|c|cccccc|}
\hline
& $1$ &$ H_2$& $H_4$& $H_{2q+2}$&$ C_P(i)$&$ C_G(i)$\\\hline
$1$ & $1$ & $1$ & $1$ & $1$ & $1$ & $1$ \\
$ H_2$ & $1$ & $1, H_2$ & $1, H_2$ & $1,H_2$ & $1$ & $H_2$ \\
$H_4$ & $1$ &$1, H_2$  & $1, H_2, H_4$  &$1,H_2, H_4$& $1$ & $H_4$   \\
$H_{2q+2}$ & $1$ &$1, H_2$  &$1,H_2, H_4$
&{ $\langle i \rangle $},{ $V,S$},$H_{2q+2} $ & $1$ & $H_{2q+2}$ \\
$ C_P(i)$ & $1$ & $1$ & $1$ & $1$ & $1,C_P(i)$ & $C_P(i)$ \\
$ C_G(i) $ & $1$ & $H_2$ & $H_4$ & $H_{2q+2}$  &$C_P(i)$  &  $C_G(i)$\\
  \hline
\end{tabular}\\
The only relevant case is, when $C_G(i)^{(x_1)},C_G(i)^{(x_2)}\simeq C_2^2\times D_\frac{q+1}{2}$. Then $[i,i^{x_1}]=[i,i^{x_2}]=1$.
 Let us define the subgroups  $V_1:=\langle i,i^{x_1}\rangle$  and
$V_2:=\langle i,i^{x_2}\rangle$.
Then we have that
  $C_G(V_l)=C_G(i)^{(x_l)}$ for $l=1,2$. Furthermore, we introduce the following notation:  $C_G(i)\geqslant C_G(V_l)=V_l\times (M_l\rtimes \langle i_l\rangle)\simeq V_l\times  D_{\frac{q+1}{2}}$ for $l=1,2$, i.e. $M_l$ denotes the cyclic
subgroups of order $\frac{q+1}{4}$ in $C_G(V_l)$.
Since  $C_G(i)\simeq \langle i \rangle \times L$, wher $L\simeq PSL(2,q)$, both
$M_1$ and $M_2$ are subgroups of the second component in this
direct product.  According to Theorem \ref{ree} (e),
 these two subgroups either coincide, or intersect trivially.
\begin{itemize}
\item[$H_{2q+2}:$]
If $M_1=M_2$ then since $C_G(M_l)=V_l\times M_l$, we have that
$V_1=V_2$ and then $C_G(i)^{(x_1,x_2)} =C_G(V_1)\simeq  H_{2q+2}$.
\end{itemize}
If they intersect trivially, then $V_1\ne V_2$ and the intersection
 $C_G(V_1)\cap C_G(V_2)=C_G(i)^{(x_1,x_2)}$ is a
$2$-subgroup. We will show that for a suitable choice of $x_1$, $x_2$ the intersection can be $S$, $V$ or $\langle i\rangle$.
\begin{itemize}
\item[$S:$] Let us choose the elements $x_1,x_2$ in such a way that
 $S:=\langle i,i^{x_1},i^{x_2} \rangle \in Syl_2(G)$, this  is possible since the involutions of $G$ are conjugate.  Let $V_1, V_2$ be as above. Then $V_1\ne V_2$, thus
$C_G(V_1)\cap C_G(V_2)$ is a $2$-group. Since it contains $S$, it must be $S$.\\
\item[$V:$] Let $V=\langle i, j \rangle $, where $i$ and $j$ are involutions in $G$ with
 $[i,j]=1$. We want to show that there are two involutions $i^{x_1}, i^{x_2}$
such that $V=C_G(i)^{(x_1,x_2)}$.
 We know that $C_G(i)\cap C_G(j)=\langle i, j\rangle\times D_{\frac{q+1}{2}}$.
Let us choose two different involutions $k_1,k_2$, from $D_{\frac{q+1}{2}}$.
Let $V_1:=\langle i, k_1\rangle, V_2:=\langle i,k_2\rangle$.
Then $V_1\ne V_2$ thus $C_G(V_1)\cap C_G(V_2)$ is a $2$-group.
However, each Sylow $2$-subgroup of $C_G(V_1)$ contains $V_1$ and
each Sylow $2$-subgroup of $C_G(V_2)$ contains $V_2$, thus the intersection cannot be a Sylow $2$-subgroup, since $[k_1,k_2]\ne 1$. So the intersection is $V$. Since   the involutions are all conjugate in $G$, thus $k_1=i^{x_1},k_2=i^{x_2}$, for suitable elements $x_1,x_2$. Hence, $V=C_G(i)^{(x_1,x_2)}$.
\item[$\langle i\rangle:$] Now we construct $V_1,V_2$ in such a way that $C_G(V_1)\cap C_G(V_2)=\langle i \rangle$.
Let $[i,i^{x_1}]=1$. Then $C_G(i)\cap C_G(i^{x_1})=\langle i \rangle\times (L\cap C_G(i^{x_1}))=\langle i\rangle \times D$, where $D\leqslant L\simeq PSL(2,q)$ is a dihedral subgroup of order $q+1$. However
by Lemma 2.2 in \cite{F}, for each dihedral subgroup $D$
 of $L\simeq PSL(2,q)$ there is an element
 $g\in L$ such that $D^g\cap D=\{1\}$. Thus
if we choose $g$ as above then $g\in C_G(i)$ and
$C_G(i)\cap C_G(i^{x_1g})=\langle i \rangle  \times {D_{q+1}}^g$ and hence
$C_G(i)^{(x_1,x_1g)}=\langle i \rangle $.
\end{itemize}
Since $H_{2q+2}$ always contains $i$, no other intersection is possible.

\end{proof}
\begin{tetel}We have that $d_c(C_G(i),G)=6$ and $d(C_G(i),G)=3$.
\end{tetel}
\begin{proof}
Now we  set up a $(\mathcal{U}_2\setminus \mathcal{U}_1)-\mathcal{U}_1$ table to study
the types of subgroups $Z\in \mathcal{U}_{\;3}(C_G(i))$.
 We know that  $Z$ is the intersection of $X$ and $Y$,
 where $X\in \mathcal{U}_{\;1}(C_G(i))$
 and $Y\in \mathcal{U}_{\;2}(C_G(i))$. Obviously the intersection of
 two elements of $\mathcal{U}_{\;1}(C_G(i))$ is in $\mathcal{U}_{\;2}(C_G(i))$,
 therefore we need to examine  only the case when Y is a new
 content of $\mathcal{U}_{\;2}(C_G(i))$. We will use the same notation
 for the subgroup types of $\mathcal{U}_{\;1}(C_G(i))$ as in the proof of
 Proposition \ref{u2cgi}. Denote the  subgroup types in 
  $\mathcal{U}_{\;2}(C_G(i))\setminus \mathcal{U}_{\;1}(C_G(i))$ by $\langle i\rangle$, $V$ and $S$, respectively,
 where $V$ is a Klein subgroup containing $i$ and $S$ is a Sylow $2$-subgroup containing $i$.\\
\begin{tabular}{|c|cccccc|}
\hline
& $1$ &$ H_2$& $H_4$& $H_{2q+2}$&$ C_P(i)$&$ C_G(i)$\\\hline
$\langle i\rangle$ & $1$ & $1$ & $1$ & $\langle i\rangle$ & $1$ & $\langle i\rangle$ \\
$ V$ & $1$ & $1, H_2$& $1,H_2$ & $\langle i\rangle, V$  & $1$ & $V$ \\
$S$ & $1$ &$1, H_2$  & $1, H_2, H_4$  &$\langle i\rangle, V,S$& $1$ & $S$   \\
  \hline
\end{tabular}\\
Now we have, that $\mathcal{U}_{\;2}(C_G(i))=\mathcal{U}_{\;3}(C_G(i))$ and thus, $d_c(C_G(i), G)$ is $5$ or $6$ by condition (1) in Definition \ref{jellemzes}.
Using this,  we will show, that the combinatorial depth is $6$.\\
Let $x_1\in C_G(i)$ and let $x_2, x_3$  be elements in $G$
such that
 $C_G(i)^{(x_2)}=C_G(\langle i,k_1\rangle )$, $C_G(i)^{(x_3)}=C_G(\langle i,k_2\rangle )$ where $i,k_1,k_2$ are involutions  as in Proposition \ref{u2cgi},
at the construction of $V$. Then $[i,k_1]=[i,k_2]=1\neq [k_1,k_2]$. 
We also have that
$i\in C_G(i)^{(x_1,x_2,x_3)}=V.$
Let us suppose that we have elements $y_1,y_2\in G$ such that
$C_G(i)^{(y_1,y_2)}=C_G(i)^{(x_1,x_2,x_3)}$ and  $x_1\in C_G(i)$ acts in the same way as $y_1$. Then $y_1\in C_G(i)$. Therefore $V=C_G(i)^{(y_1,y_2)}=C_G(i)^{(y_2)}\in \mathcal{U}_1(C_G(i))$, which is a contradiction.
Thus, $d_c(C_G(i),G)=6$. 
Finally, by Theorem \ref{inter} we get that $d(C_G(i),G)=3$.
\end{proof}

\section{The depth of $G_0$, if $G_0\not\simeq R(3)$}
In this section we consider  a maximal subgroup $G_0\simeq R(q_0)$ of $G$, where
 $q_0=3^{2n_0+1}=3m_0^2$ for $n_0\geq 1$. We will study  which subgroups
of $G$ can occur in the form $G_0^{(g)}\ne G_0$, by considering the maximal subgroups of $G_0$ that can contain them. 
Let $H$ be a maximal subgroup of $G_0$. Let us introduce the notation
  $U_{H}=\{G_0^{(g)}\ |\ g\in G\setminus G_0 \mbox{ and } G_0^{(g)}\leqslant H\}$. Then we have that
$\mathcal{U}_{\;1}(G_0)=G_0 \bigcup  \cup_{H\in Max(G_0)}\ U_{H}.$

The following Lemma shows that  in several cases $G_0^{(g)}$ is already in $\mathcal{U}_{\;1}(H)$, for some $H<G_0$.

\begin{lemma}{\label{abcd}}
\begin{enumerate}[a)]
\item Let $H_0,T_0\leqslant G$ and  $g\in G$. If $H_0\leqslant T_0^{(g)}$,
 and the subgroups of $T_0$ of order $|H_0|$ are conjugate in $T_0$, then $g\in T_0N_G(H_0)$.
\item  Let $T_0,H\leqslant G$ and  $g\in G$. If there is an element $t\in T_0$ such that $t^{-1}g\in N_G(H)$ and $T_0^{(g)}\leqslant H$, then $T_0^{(g)}=(T_0\cap H)^{(t^{-1}g)}$.
\item Let $H\leqslant G$. If there exists an element $r\in G_0$ such that for a fixed $g\in G$, $r^{-1}g\in N_G(H)$, then $G_0^{(g)}\cap N_G(H)=(N_{G_0}(H))^{(r^{-1}g)}$.
\item Let $H_0\leqslant G_0$ such that $N_{G_0}(H_0)$ is the only maximal subgroup in $ G_0$ containing $H_0$. If there exists an element $r\in G_0$ such that for a fixed $g\in G\backslash G_0$, $r^{-1}g\in N_{G}(H_0)$, then $G_0^{(g)}=N_{G_0}(H_0))^{(r^{-1}g)}$.
\end{enumerate}
\end{lemma}
\begin{proof}
\begin{enumerate}[a)]
\item Since $H_0^{g^{-1}}\leqslant T_0$, there is an element $t\in T_0$ such that $H_0^{g^{-1}t}=H_0$ and $g^{-1}t\in N_G(H_0)$.
\item It can be easily checked.
\item $G_0^{(g)}\cap N_G(H)=G_0^{(r^{-1}g)}\cap N_G(H)=N_{G_0}(H)\cap N_{G_0^{r^{-1}g}}(H)=N_{G_0}(H)\cap (N_{G_0}(H))^{r^{-1}g}$
\item It is easy  to see that  $G_0^{(g)}\geqslant (N_{G_0}(H_0))^{(r^{-1}g)}$,
thus  $H_0\leqslant G_0^{(g)}$. By assumption, $G_0^{(g)}\neq G_0$.
 Since there is only one maximal subgroup containing $H_0$ in $G_0$ and $G_0^{r^{-1}g}$,
  respectively,
we have that $G_0^{(g)}\leqslant N_{G_0}(H_0)\cap  N_{G_0^{r^{-1}g}}(H_0)=(N_{G_0}(H_0))^{(r^{-1}g)}$.
\end{enumerate}
\end{proof}
 The following Lemma shows  that  if we can find elements of certain type   in
 $G_0^{(g)}$, then   their centralizer in $G_0$ also lies in $G_0^{(g)}$.

 \begin{lemma}\label{6tolrelpr}
\begin{enumerate}[a)]
\item If a non-trivial element $x\in G$, whose order divides $\frac{q_0+1}{4}$,
    is  contained in $G_0^{(g)}$, then $C_{G_0}(x)(\simeq C_2^2\times C_{\frac{q_0+1}{4}})\leqslant G_0^{(g)}$.
\item If a non-trivial  element $x\in G$, whose order divides $\frac{q_0-1}{2}$, is contained in $G_0^{(g)}$,  then $C_{G_0}(x)(\simeq C_{q_0-1}) \leqslant G_0^{(g)}$.
\item If a non-trivial element $x\in G$, whose order divides either $q_0+ 3m_0+1$ or $q_0- 3m_0+1$, is contained  in $G_0^{(g)}$, then $C_{G_0}(x)\leqslant G_0^{(g)}$.
(Here
$C_{G_0}(x)$  is isomorphic to $ C_{q_0+ 3m_0+1}$ or $C_{q_0- 3m_0+1}$, respectively.)
\item If  a Klein subgroup $V$  is contained in $G_0^{(g)}$,
then the unique subgroup  $M_0$ of order $\frac{q_0+1}{4}$ of
$C_{G}(V)= V\times (M\rtimes C_2)$  is a subgroup of $G_0^{(g)}$, thus $C_{G_0}(M_0)=V\times M_0\leq G_0^{(g)}$.
\end{enumerate}
\end{lemma}
\begin{proof}
\begin{enumerate}[a)]
\item By Lemma \ref{centralOfqpm1} we know that $C_{G_0}(x)$, $C_{G_0^g}(x)\simeq C_2^2\times C_{\frac{q_0+1}{4}}$ and $C_{G}(x)\simeq C_2^2\times C_{\frac{q+1}{4}}$. Since $C_G(x)$ contains only one subgroup of order $q_0+1$, we have that $C_{G_0}(x)= C_{G_0^g}(x)\leqslant G_0^{(g)}$. 
\item By Theorem \ref{centralOfqpm1} we know, that both $C_{G_0}(x)$ and
 $C_{G_0^g}(x)$ are isomorphic to $C_{q_0-1}$ and $C_{G}(x)\simeq C_{q-1}$.
 Since $C_G(x)$ contains only one subgroup of order $q_0-1$, we have that
 $C_{G_0}(x)=C_{G_0^g}(x)\leqslant G_0^{(g)}$. 
\item We consider the case, when the order of $x$ divides $q_0+3m_0+1$, the other case is similar. By  Theorem \ref{ree} (d)
 we know that $C_{G_0}(x)$, $C_{G_0^g}(x)\simeq C_{q_0+3m_0+1}$ and $C_G(x)=C_2^2\times C_{\frac{q+1}{4}}$ or $C_{3q\pm 3m+1}$ depending on, whether $\frac{q+1}{4}$ or $q\pm3m+1$ is divisible by $q_0+3m_0+1$. In each  case there is only one subgroup of order $q_0+3m_0+1$ in $C_G(x)$ and  hence $C_{G_0}(x)=C_{G_0^g}(x)\leqslant G_0^{(g)}.$
\item Let us suppose that $G_0^{(g)}$ contains a Klein subgroup $V$.
  Observe that the subgroup $M_0$ of order $\frac{q_0+1}{4}$ of $C_G(V)=V\times (M\rtimes C_2)$ is unique. However,
$C_{G_0}(V)$ and $C_{G_0^g}(V)$ are both isomorphic to
 $V\times (C_{\frac{q_0+1}{4}}\rtimes C_2)$ and they are contained in $C_G(V)$.
Thus $M_0\leq G_0^{(g)}$, and hence $C_{G_0}(M_0)=V\times M_0\leq G_0^{(g)}$.
\end{enumerate}
\end{proof}

Now we determine $U_{N_{G_0}(P_0)}$ for a fixed Sylow
$3$-subgroup $P_0$ of $G_0$.

\begin{lemma}\label{hat}
 Let $P$ be a Sylow $3$-subgroup of the Ree group $G$ and let $W$ be an arbitary complement of $P$ in $N_G(P)$. Let $W_{2'}$ be the $2'$ part of $W$. Then:
\begin{enumerate}[a)]
\item Every nontrivial element of $W$ acts regularly on $Z(P)\setminus \{1\}$ via the conjugation action.
\item The  conjugation action of every nontrivial element $w\in W_{2'}$ on $P'/ Z(P)$ has only one fixed point, namely $Z(P)$. Furthermore, the conjugation action of $W_{2'}$ has 3 orbits on $P'/Z(P)$, where $O_1=Z(P)$ and $O_2^{-1}=O_3$.
\item The  conjugation action of every nontrivial element $w\in W_{2'}$ on $P/ P'$ has only one fixed point, $P'$.
\end{enumerate}
\end{lemma}
\begin{proof} Obviously $Z(P)$ and $P'$ are $W$-invariant.
\begin{enumerate}[a)]
\item By Theorem \ref{ree} (b), for every nontrivial element
$w\in W_{2'}$, 
we have that $C_{Z(P)}(w)=\{1\}$. 
  The same holds for order $2$ elements of $W$.
 Hence, $C_{Z(P)}(w)=1$, for every  nontrivial element $w\in W$.
 Since $|W|=|Z(P)\setminus \{1\}|$, the action is regular.
\item By Theorem \ref{ree} (b), for every
 nontrivial element $w\in W_{2'}$, we have that  $C_{P'}(\langle w\rangle)=\{1\}$.
 
  Using  \cite[Theorem 3.15, Ch. 5, p. 187]{G},
 we have that  $C_{P'/Z(P)}(\langle w\rangle )=\{\overline 1\}$.  Obviously,
 an element and its inverse cannot belong to the same $W_{2'}$-orbit.
Since $|W_{2'}|=\frac{|P'/Z(P)|-1}{2}$, we are done.
\item We use again \cite[Theorem 3.15, Ch. 5, p.187]{G} to deduce that
$C_{P/P'}(\langle w \rangle)=\{\overline 1\}$.
\end{enumerate}
\end{proof}
\begin{prop}\label{3Sly}
Let $G$ be a Ree group, $G_0$ a Ree-subgroup of $G$ and $P$ a Sylow
 $3$-subgroup of $G$. Then $P\cap G_0$ is either trivial or a Sylow $3$-subgroup of $G_0$.
\end{prop}
\begin{proof}
Assume, that the intersection of $G_0$ and $P$ is not trivial.
 In this case there is a Sylow $3$-subgroup $P_0$ of $G_0$,  containing
 the intersection. Since the Sylow $3$-subgroups in $G_0$ are TI,
 $P_0$ is unique. If $P$ would not contain $P_0$, then there would be
 another Sylow $3$-subgroup $R$ of $G$, which would contain it. However,
 in this case $P$ and $R$ would have common elements, and this would contradict
 the fact, that the Sylow $3$-subgroups of $G$ are TI.
Thus, we have that $G_0\cap P$ contains $P_0$ and hence they are equal.
\end{proof}
\begin{notation}\label{3jel} We fix the following notation.
 Let $G_0$ be a  fixed  maximal subgroup of $G$ isomorphic to $R(q_0)$. Let
$P_0$ be a Sylow $3$-subgroup of $G_0$. Let $P\in Syl_3(G)$  containing $P_0$.
Let us suppose that  in the representation of $P$ of Theorem \ref{ree} 
(h), the
 elements
of $P_0$ correspond to triples $\{(a,b,c) | a,b,c\in GF(q_0) \}$.
 Let $W_0$ be a complement of $P_0$ in $N_{G_0}(P_0)$  let $W$
be  the complement of $P$ in $N_G(P)$ containing $W_0$. Let $i$ be the unique
involution in $W_0$.
\end{notation}
\begin{lemma}\label{*} We use  Notation \ref{3jel}.
\begin{enumerate}[a)]
\item Let $x\in G$
 and  let $j$ be an involution in $P_0W_0$. Then  $j\in (P_0W_0)^{(x)}$ if and only if $x\in P_0 C_G(j)$.

\item  Let $W_1\le P_0W_0$ of order $q_0-1$. Then $W_1\leq (P_0W_0)^{(x)}$
if and only if $x\in P_0N_G(W_1)$.
\end{enumerate}
\end{lemma}
\begin{proof} a)$\Rightarrow$ By assumption $j,j^{x^{-1}}\in P_0W_0$. Since the involutions are conjugate in $P_0W_0$,
 we have that
 there exists an element $p_0\in P_0$ such that $j^{p_0}=j^{x^{-1}}$, thus
 $p_0x\in C_G(j)$ and $x\in P_0C_G(j)$.\\
    $\Leftarrow$ By assumption there is an element $p_0\in P_0$ such that $p_0x\in C_G(j)$, thus $(P_0W_0)^{(x)}=(P_0W_0)^{(p_0x)}$, which  contains $j$.

b) The proof is similar to that of a).
\end{proof}
\begin{prop}\label{ungp} Using Notation \ref{3jel} we have that
\[U_{N_{G_0}(P_0)}=\{ Z(P_0),P_ 0', P_0'\langle i^{p_0}\rangle\ |p_0\in P_0\}\cup U\cup V,\]
where $V\subseteq \{Z(P_0)\langle h\rangle, P_0,P_0\langle i\rangle, P_0W_0 |\ h\in P_0\setminus P_0', p_0\in P_0\}$ and $U\subseteq \{ 1, \langle i^{p_0}\rangle , W_0^{p_0} \ | \  p_0\in P_0\}$.
\end{prop}
\begin{proof}It directly follows from:
\end{proof}
\begin{lemma}\label{ungpcomp}Let $g\in G$ be an element
 such that $G_0^{(g)}\ne G_0$.  Then
\begin{enumerate}[(A)]
\item If $G_0^{(g)}$ contains a $3$-element $p\in P_0$,
then there exists an element $r\in G_0$ such that $r^{-1}g\in N_G(P)$.
 Moreover, this element $r$ can be chosen so that $r^{-1}g=p_1\in P\setminus P_0= (Z(P)\setminus Z(P_0))\cup (P'\setminus P_0'Z(P))\cup (P\setminus  P_0P')$ and $G_0^{(g)}=(P_0W_0)^{(r^{-1}g)}=(P_0W_0)^{(p_1)}$. In particular, $Z(P_0)\leqslant G_0^{(g)}$
and $G_0^{(g)}\leq N_{G_0}(P_0)$.
    \begin{enumerate}[I.]
    \item{ If $r^{-1}g\in Z(P)\setminus Z(P_0)$ then $G_0^{(g)}=P_0,P_0\langle i \rangle$ or $P_0W_0$.}
\item{ If $r^{-1}g\in P'\setminus P_0'Z(P)$ then $G_0^{(g)}$
 is isomorphic to $P_0'$ or $P_0'\langle i \rangle$.
 Moreover, there exists an element $p_1  \in  P'\setminus P_0'Z(P)$  such that $G_0^{(p_1)}=P_0'$ and for every involution $j\in P_0W_0$ there exists
an element $p_2 \in  P'\setminus P_0'Z(P)$ such that
   $G_0^{(p_2)}=P_0'\langle j\rangle$. The case $P_0'W_0$ does not occur.}
\item{ If $r^{-1}g\in P\setminus P_0P'$ then $G_0^{(g)}$ is $Z(P_0)$ or $ Z(P_0)\langle h\rangle$, where $h$ is an element of order $9$ in $P_0$.
 Moreover, there exists an element $p_1\in P\setminus P_0P' $ such that $G_0^{(p_1)}=Z(P_0)$.}

\end{enumerate}
\item If $G_0^{(g)}\leq N_{G_0}(P_0)$ and  $G_0^{(g)}$ does not contain any nontrivial $3$-elements from $P_0$, then $G_0^{(g)}$ is isomorphic  to $1,C_2$ or  $C_{q_0-1}$.
\end{enumerate}
\end{lemma}
\begin{proof}
{\it Case (A)} Let $p\in G_0^{(g)}$ be a nontrivial $3$-element belonging to $P_0$.
 Let $S_0\in Syl_3(G_0)$ containing $p^{g^{-1}}$.
Then there exists an element $r\in G_0$ such that ${S_0}^r=P_0$. Since $p,p^{g^{-1}r}\in P_0\leq P$, and $P$ is TI, we have that
$r^{-1}g\in N_G(P)$. Thus  $g\in G_0N_G(P)\setminus G_0$.
 Then, since
 $G_0WP=G_0W_0P\cup G_0(W\setminus W_0)P= G_0(P\setminus P_0)\cup G_0\cup G_0(W\setminus W_0)P$,
 we have that
$G_0N_G(P)\setminus G_0= G_0(P\setminus P_0) \cup G_0(W\setminus W_0)P.$
\indent
We will show that if $g\in G_0(W\setminus W_0)P$ then $G_0^{(g)}$ does not contain
 any non-trivial elements from $P_0$. Thus in Case (A) $g\in G_0(P\setminus P_0)$
must hold.

Let $g=rw_1 p_1$, where $r\in G_0$, $w_1\in W\setminus W_0$ and $p_1\in P $.
We may suppose that $w_1\in W_{2'}$, otherwise, since $g=(ri)(iw_1)p_1$,
we choose $ri\in G_0$ instead of $r$ and $iw_1\in W\backslash W_0$ instead of $w_1$.
Using the fact that $N_{G_0}(P_0)\leq N_{G_0}(P)$ and
  Lemma \ref{abcd} c), with  $H:=P$ and $r:=r$,
 we have that $G_0^{(g)}\cap N_G(P)=
 N_{G_0}(P)^{(w_1p_1)}=N_{G_0}(P_0)^{(w_1p_1)}=(P_0W_0)^{(w_1p_1)}$.\\
Let $1\ne z\in Z(P_0)\cap (P_0W_0)^{w_1p_1}$. Then $z\in (P_0W_0)^{w_1}$
and $z^{w_1^{-1}}\in P_0\cap Z(P)=Z(P_0)$. Let $z_0=z^{w_1^{-1}}$. Then there exists an element $w_0\in W_0$ such that $z^{w_0}=z_0$. Hence $z^{w_0w_1}=z$.
Using Lemma \ref{hat} a) we have a contradiction.
 Hence $(P_0W_0)^{(w_1p_1)}$  does not
 contain  any nontrivial elements from $Z(P_0)$.\\
    Hence $P_0W_0^{(w_1p_1)}$  neither contains any elements of order $9$
 from $P_0$. We show  that $(P_0W_0)^{(w_1p_1)}$ does not contain any other $3$-elements.  Let $h$ be an element of order $3$ in
 $(P_0W_0)^{(w_1p_1)}$. Then $h \in P_0'\backslash Z(P_0)$.
    Using the representation of the Sylow $3$-subgroups $P_0\in Syl_3(G_0)$ and $P\in Syl_3(G)$, we have that
  $h=(0,y_0,z_0)$ and $p_1^{-1}=(a,b,c)$, where
 $y_0\neq 0$, $y_0,z_0\in GF(q_0)$ and $a,b,c\in GF(q)$. Denote $w_1^{-1}$ by $w$. Then by Theorem  \ref{ree} (h) and (by Lemma \ref{conj}) we have that
    $h^{p_1^{-1}w}=(0,y_0,z_0)^{w}(0,0,2y_0a)^{w}\in P_0\cap P'=P_0'.$
 Denote the element  $(0,0,2y_0a)^w$ by $\zeta$. Since $(0,y_0,z_0)^w\zeta\in P_0'\setminus Z(P_0)$, applying Lemma \ref{hat} b) for $P_0'$ we have that there are elements $w_0\in (W_0)_{2'}$ and $\zeta_0\in Z(P_0)$ such that $(0,y_0,z_0)^w\zeta=(0,y_0,z_0)^{w_0}\zeta_0$ or $(0,y_0,z_0)^w\zeta=((0,y_0,z_0)^{-1})^{w_0}\zeta_0$.
    Using Lemma \ref{hat} b) again for $P^{'}$, we have in both cases a contradiction. In the first case $ww_0^{-1}\in W_{2'}$ fixes $(0,y_0,z_0)Z(P)$, and in the second case $ww_0^{-1}$ takes $(0,y_0,z_0)Z(P)$ to its inverse.
    Thus $G_0^{(g)}\cap P_0=\{1\}$,  which is a contradiction. Thus
$g\not\in G_0(W\backslash W_0)P$.
\\

 Hence if $G_0^{(g)}$ contains  a nontrivial
 $3$-element $p\in P_0$, then $g\in G_0(P\setminus P_0)$ i. e. $g=rp_1$ where $r\in G_0$ and $p_1\in P\setminus P_0$. Thus by Lemma \ref{abcd} d) with $Z(P_0)$ as $H_0$, we have
 that $G_0^{(g)}=(N_{G_0}(Z(P_0)))^{(p_1)}=(P_0W_0)^{(p_1)}$.\\
Obviously, $P\setminus P_0$ is a disjoint union of $Z(P)\setminus Z(P_0)$, $P'\setminus P_0'Z(P)$ and $P\setminus P_0 P'$. Hence $p_1=r^{-1}g$ belongs to one of them.
\begin{enumerate}[I.]
\item{If $r^{-1}g=p_1\in Z(P)\setminus Z(P_0)$ then as $[P_0,p_1]=1$,
$G_0^{(g)}=(P_0W_0)^{(p_1)}\geqslant P_0$. Using  Lemma \ref{6tolrelpr} b) we have that $G_0^{(g)}$ is  $P_0,P_0\langle i \rangle $ or $P_0W_0$.
}

\item{If $r^{-1}g=p_1\in P'\setminus P_0' Z(P)$, then $G_0^{(g)}=(P_0W_0)^{(p_1)}\geqslant P_0'$, since $P'$ is elementary abelian.
If $G_0^{(g)}=(P_0W_0)^{(p_1)}$ would contain an element $h$
of order $9$, then $h^{{p_1}^{-1}}\in P_0$ would hold. Using the representation
of elements of the Sylow $3$-subgroup of $G$ in Theorem \ref{ree} (h) and
Lemma \ref{conj}, we have that
if ${p_1}^{-1}=(0,b,c)$ and $h=(x_0,y_0,z_0)$ with $b\not\in GF(q_0)$,
 $x_0\ne 0$ and $x_0,y_0,z_0 \in GF(q_0)$, then $h^{{{p_1}}^{-1}}=(x_0,y_0,z_0-2x_0b)\in P_0$, which is a contradiction.
Hence, again by Lemma \ref{6tolrelpr} b), $G_0^{(g)}$ can only be isomorphic
 to  $P_0',P_0'\rtimes C_2 $ or  $P_0'\rtimes C_{q_0-1}$.}

Now we prove that $G_0^{(g)}=(P_0W_0)^{(p_1)}$ is not isomorphic to $P_0'\rtimes C_{q_0-1}$. Assume that $W_1=W_0^{p_0}\leqslant (P_0W_0)^{(p_1)}$ for some $p_0\in P_0$. By Lemma \ref{*} b) $W_1\leqslant (P_0W_0)^{(p_1)}$ if and only if $p_1\in P_0N_G(W_1)$. Thus  we have that $p_1\in P_0N_G(W_1)$, in particular, $p_1\in P_0N_P(W_1)$. However by Theorem \ref{ree} (b)  we have that
 $N_G(W_1)=W_1\rtimes C_2$,  hence $N_P(W_1)=1$. Thus $p_1\in P_0$,  which contradicts our assumption, that $p_1\in P'\setminus P_0'Z(P)$.\\

Now we show, that for every involution $j\in P_0W_0$ if we take $p_1\in C_{P\setminus P_0}(j)$, which is nontrivial,  then $G_0^{(p_1)}=P_0'\langle j\rangle$. By Theorem \ref{ree} (b) we have that
 $p_1\in P'\setminus Z(P)$ and $p_1\in P'\setminus P_0'$. We show that $p_1\in P'\setminus P_0'Z(P)$. For this, consider the representation of the Sylow $3$-subgroup of $G$ in Theorem \ref{ree} (h).
By Theorem \ref{ree} (b) we have that $C_P(j)\cap Z(P)=\{1\}$. We prove that
 for every $b\in GF(q)$ there is at most one $c\in GF(q)$ such that $(0,b,c)\in C_P(j)$.
Let us suppose that  $(0,b,c_1),(0,b,c_2)\in C_P(j)$.  Then their quotient, $(0,0,c_1-c_2)\in C_P(j)\cap Z(P)=\{1\}$. Hence  $c_1=c_2$.
 Since $|C_P(j)|=q$,   for every $b\in GF(q)$ there is exactly
 one $c\in GF(q)$  with $(0,b,c)\in C_P(j)$. Similar statement holds for
$C_{P_0}(j)$.  Hence, if $b\in GF(q_0)$, then the
unique $c$ must be in $GF(q_0)$ and so $C_{P_0'Z(P)}(j)=C_{P_0'}(j)$.
  Thus, $p_1\in (P'\setminus P_0')\cap C_{P}(j)$ implies that
 $p_1\in P'\setminus P_0'Z(P)$. Since $p_1\in N_G(P)$,  we have seen at the beginning of the proof
that
 $G_0^{(p_1)}=(P_0W_0)^{(p_1)}$. Using that $p_1\in P'\setminus P_0'Z(P)$, we proved already that $(P_0W_0)^{(p_1)}$ is isomorphic to $P_0'$ or
$P_0'\langle i \rangle$. By Lemma \ref{*} we have that $j\in (P_0W_0)^{(p_1)}$.
Hence
$G_0^{(p_1)}=P_0'\langle j\rangle$.\\

Finally we prove, that there is an element $p_1\in P'\setminus P_0' Z(P)$ such that $(P_0W_0)^{(p_1)}$ does not contain any involutions.
 Let $j\in (P_0W_0)^{(p_1)}$ be a fixed involution.
 Using Lemma \ref{*}  for $x=p_1\in P'\setminus P_0'Z(P)$,
 by Theorem \ref{ree} (b) and (h) we have that\\
    $|\{p_1\in P'\setminus P_0' Z(P)\ |\ j\in(P_0W_0)^{(p_1)} \}|= |(P'\setminus P_0' Z(P))\cap (P_0C_P(j))|$ $\leq|P_0'C_P(j)\setminus P_0'|=(q-q_0)q_0$.

    If for some element $p_1\in P'\setminus P_0' Z(P)$, the subgroup
$(P_0W_0)^{(p_1)}$ contains an involution $j$, then
 $(P_0W_0)^{(p_1)}=P_0'\langle j \rangle $.
 Since the involutions are conjugate in $P_0W_0$ by elements of $P_0$,
 for every involution $j\in P_0W_0$ the same number of $p_1\in P'\setminus P_0'Z(P)$ occurs such that $(P_0W_0)^{(p_1)}$ contains $j$. Since by Lemma \ref{centralizerof3orderelement} we  have that $C_P(p_1)=P'$, the cosets of $P_0'$
in $P_0$ move $p_1$ to $q_0$ different places.
Thus, from the number of elements $p_1$ belonging to  the involution $j$, one can get
 an upper bound on  the number of elements $p_1$ belonging
to any involution in $P_0W_0$, by multiplying with $q_0$.
Hence, we have that
    $|\{p_1\in P'\setminus P_0' Z(P)\ |\ (P_0W_0)^{(p_1)}\mbox{ contains involutions} \}|\leq (q-q_0)q_0^2.$
    Since $|P'\setminus P_0' Z(P) |=q(q-q_0)$ is bigger than
 $q_0^2 (q-q_0)$, there is an element $p_1\in P'\setminus P_0' Z(P)$
 such that $(P_0W_0)^{(p_1)}$ does not contain involutions.
Then, by the beginning of the proof of II.,  we have that $(P_0W_0)^{(p_1)} =P_0'$  and hence
$G_0^{(p_1)}=(P_0W_0)^{(p_1)}=P_0'$.\\

\item{If $r^{-1}g=p_1\in P\setminus {P_0P'}$ then $G_0^{(g)}=(P_0W_0)^{(p_1)}\geqslant Z(P_0)$. 
We show that $(P_0W_0)^{(p_1)}$ cannot contain noncentral elements of order $3$ in $P_0$.}
As before, we use the representation in Theorem \ref{ree} (h) of $P$ and $P_0$.
Let us suppose that $h\in (P_0W_0)^{(p_1)}$ is a noncentral element of order $3$
in $P_0$. Let $p_1^{-1}=(a,b,c)$ and $h=(0,y_0,z_0)\in P_0$. Then
$a\not\in GF(q_0)$, otherwise $(a,b,c)=(a,0,0)(0,b,c+ab)\in P_0P'$, which is not the case. On the other hand,   $y_0\ne 0$, since otherwise  $h\in Z(P_0)$.
Then by Lemma \ref{conj}, $h^{{p_1}^{-1}}=(0,y_0,z_0+2y_0a)\in P_0$, which is a contradiction.\\

Now we prove that $G_0^{(g)}=(P_0W_0)^{(p_1)}$ does not contain any elements outside $P_0$. As before, using Lemma \ref{6tolrelpr} b), it is enough to show that $(P_0W_0)^{(p_1)}$ does not contain involutions.\\
Suppose this is not true and let $j$ be an involution in $(P_0W_0)^{(p_1)}$.
By Lemma \ref{*} we have  that $p_1\in P_0C_P(j)$ and by Theorem \ref{ree} (b)
  we have that $p_1\in P_0P'$, which  contradicts  our assumption.\\

Now we show that if  there is an element $h\in P_0$ of order $9$
in $(P_0W_0)^{(p_1)}$ then $(P_0W_0)^{(p_1)}= Z(P_0)\langle h\rangle$.
 We use the representation of the Sylow $3$-subgroups $P_0$ and $P$, as before.
Let ${p_1}^{-1}=(a,b,c)$ and $h=(x_0,y_0,z_0)$. Clearly $a\not\in GF(q_0)$, $x_0,y_0,z_0\in GF(q_0)$ and $x_0\ne 0$. Then by Lemma \ref{conj},
 we have that\\
 $h^{{p_1}^{-1}}=(x_0,y_0 +x_0(a\sigma)-a(x_0\sigma),z_0-2x_0b+2y_0a-ax_0(x_0\sigma)+ax_0(a\sigma))\in P_0.$
Consider the solutions  $x\in GF(q_0)\setminus \{0\}$ of the relation
 $x(a\sigma)-a(x\sigma)\in GF(q_0)$. Applying $\sigma $ to this relation, as $\sigma $ is also an automorphism of $GF(q_0)$, we have that
$(x\sigma)a^3-(a\sigma)x^3\in GF(q_0)$. Multiplying by $x^2$ on the
 left hand side of the original relation, adding the left hand side of the second to it and dividing by $(x\sigma)$, we have that
$a^3-ax^2\in GF(q_0)$. Recall that $x_0$ and $-x_0$ both satisfy this relation.
If some $x\in GF(q_0)\setminus \{0\}$ also satisfies it, then
 $a^3-a(x_0)^2-(a^3-ax^2)=a(x_0-x)(x_0+x)\in GF(q_0)$. Since $a\not\in GF(q_0)$, we have that either $x=x_0$ or $x=-x_0$.
Let $h'=(x',y',z')\in (P_0W_0)^{(p_1)}$ be another element. As we have seen,
 $h'\in P_0$.
 If $h'\notin Z(P_0)$, then it is of order $9$. 
Then
$h'^{{p_1}^{-1}}=(x',y'+x'(a\sigma)-a(x'\sigma),z'-2x'b+2y'a-ax'(x'\sigma)+ax'(a\sigma))\in P_0.$

Then $x'$ also satisfies $a^3-ax'\in GF(q_0)$, thus $x'=x_0$ or $x'=-x_0$ holds.

If $x'=x_0$ then by subtracting the $3$-rd
component of $h'^{{p_1}^{-1}}$ from that of $h^{{p_1}^{-1}}$ we have that
$2a(y_0-y')\in GF(q_0)$. Since $a\not\in GF(q_0)$ we have that
$y_0-y'=0$, hence
$h'\in hZ(P_0)$.\\ If $x'=-x_0$ then by adding
 the third component
 of $h'^{{p_1}^{-1}}$ to that of $h^{{p_1}^{-1}}$, we have that
$2a(y_0+y'-x_0(x_0\sigma)\in GF(q_0)$.
Since $a\not\in GF(q_0)$, we have that $y'=-y_0+x_0(x_0\sigma )$, thus by Lemma \ref{conj}, we have that $h'\in h^{-1}Z(P_0)$ and hence
$(P_0W_0)^{(p_1)}$ is  equal to  $ Z(P_0)\langle h \rangle$.\\

Thus we have seen that $p_1\in P\setminus P_0P'$ implies that
 $(P_0W_0)^{(p_1)}$ can be either $Z(P_0)\langle h\rangle$ for
 some element $h\in P_0$ of order $9$, or $(P_0W_0)^{(p_1)}=Z(P_0)$.

We show that there is an element $p_1\in P\setminus {P_0P'}$ such that $G_0^{(p_1)}=(P_0W_0)^{(p_1)}=Z(P_0)$.\\
Let $p_1^{-1}=(a,b,c)\in P\setminus P_0P'$.
Let $h=(x_0,y_0,z_0)\in P_0$ be an element such that $x_0\ne 0$, i.e. it is of order $9$, and let $h\in (P_0W_0)^{(p_1)}$. Then $h^{p_1^{-1}}\in P_0$ hence
$x_0,y_0,z_0$ satisfies

\begin{gather}x_0(a\sigma )-a(x_0\sigma )\in GF(q_0)\tag{1}
\end{gather}
and
\begin{gather}\tag{2}
-2x_0b+2y_0a-ax_0(x_0\sigma )+ax_0(a\sigma )\in GF(q_0)
\end{gather}

 Let $(x_0,y_0,z_0)\in P_0$ be fixed such that $x_0\neq 0$. Then\\
$|\{p_1\in P\setminus P_0P' \ |\ (x_0,y_0,z_0)\in (P_0W_0)^{(p_1)}\}|$$\leq
|\{(a,b,c)\in GF(q)^3\ |\  a\notin GF(q_0), \mbox{ and b satisfies (2)}\}|$
$\leq (q-q_0)q_0q.$\\
If $p_1\in P\setminus P_0P'$ and $(P_0W_0)^{(p_1)}$ contains an element $h$ of order $9$, then $(P_0W_0)^{(p_1)}= Z(P_0)\langle h \rangle $. It
 contains exactly $2q_0$ elements of order $9$, since $h^3\in Z(P_0)$.

So there are $\frac{q_0^3-q_0^2}{2q_0}$  elements of order $9$ in $P_0$ giving
different subgroups in $P_0$ isomorphic to $ Z(P_0)\langle h \rangle$. Thus\\
$|\{p_1\in P\setminus P_0P' \ |\ (P_0W_0)^{(p_1)}\mbox{ contains elements of order $9$}\}| \leq \frac{q_0^3-q_0^2}{2q_0} q_0q(q-q_0).$\\
Since $|P\setminus P_0P'|=q^2(q-q_0)$ is bigger than $\frac{1}{2} q (q_0^3-q_0^2)(q-q_0)$, we have
 that there is an element $p_1$ such that $G_0^{(p_1)}=(P_0W_0)^{(p_1)}=Z(P_0).$

\end{enumerate}
{\it Case (B)}
Since $G_0^{(g)}$ does not contain $3$-elements and $G_0^{(g)}\leqslant N_{G_0}(P_0)=P_0W_0$, by Lemma \ref{6tolrelpr} b) we have that $G_0^{(g)}$ is isomorphic  to $1,C_2$ or  $C_{q_0-1}$.
\end{proof}

Now we introduce the following notation to determine $U_{N_G(M^{+1})}$, $U_{N_G(M^{-1})}$ and $U_{N_G(M)}$.

\begin{notation}\label{Hjel}
Let $G_0$ be a maximal subgroup of $G$ isomorphic to $R(q_0)$. Let
$M_0\in Hall_{\frac{q_0+1}{4}}(G_0)$.
 Since $q=q_0^a$, where $a$ is an odd prime, thus  $q_0+1|q+1$,
and hence $M_0\leq M$ for some $M\in Hall_{\frac{q+1}{4}}(G)$.
By Theorem \ref{ree} (e)  we have that $C_{G_0}(M_0)=M_0\times V_0$
and $C_G(M)=M\times V$, for some Klein-subgroups $V_0$ and $V$.
Since $V\triangleleft N_G(M)=V\times (M\rtimes C_2)\rtimes C_3$, we have that $V$ is contained in every Sylow $2$-subgroup of $N_G(M)$. Since by Theorem \ref{ree} (f), $N_{G_0}(M_0)\leq N_G(M)$, at least one of these Sylow $2$-subgroups is inside $N_{G_0}(M_0)$. Hence $V\leq N_{G_0}(M_0)$.
Using    $[V,M_0]=1$,  we have that $V\leq C_{G_0}(M_0)$ and thus $V_0=V$.
 We have that $N_{G_0}(M_0)=(M_0\times V_0)\rtimes \langle t\rangle$, where the order of  element $t$ is $6$. Since by Theorem \ref{ree} (f),
 $N_G(M_0)\leq N_G(M)$, there is equality here.
Hence we have that  $N_{G_0}(M_0)\leq N_G(M_0)=N_G(M)=(M\times V_0)\rtimes \langle t\rangle$.\\
 Let  $M_0^{j}\leq G_0$ be a  Hall subgroup of order $q_0+1+3jm_0$, where
$j=\pm 1$. It can be embedded to a Hall subgroup $\tilde M^{j'}$ of order $q+1+j'3m$ in $G$ or to a Hall subgroup $\tilde M$ of order $\frac{q+1}{4}$
 depending on which factor of $\frac{q^3+1}{4}=\frac{q+1}{4}(q+1+3m)(q+1-3m)$ is divisible by $q_0+1+j3m_0$. )
 \\
Furthermore, similarly as in the case of  $M_0$, by
Theorem \ref{ree} (d) (e) and (f), $N_{G_0}(M_0^j)=M_0^j\rtimes \langle t \rangle$  for
 $j=\pm 1$ and for an element $t$ of order $6$, and $N_G(M_0^{j})=N_G(\tilde M^{j'})$, where $j'=\pm j$
 or $N_G(M_0^{j})=N_G(\tilde M)$, depending on if $M_0^j\leq \tilde M^{\pm j}$ or
$M_0^j\leq \tilde M$. Thus\\
$N_G(M_0^j)\in \{\tilde M^{+ 1}\rtimes \langle t \rangle \mbox{ , } \tilde M^{-1}\rtimes \langle t \rangle\mbox{ , }(\tilde M\times V)\rtimes \langle t\rangle\}.$ 
 \end{notation}

\begin{prop}\label{ungm+} Using Notation \ref{Hjel} we have that
if  $G_0^{(g)}$ contains a nontrivial element from $M_0^{+1}$ then $g\in G_0  N_G(M_0^{+1})$. In particular\\
$U_{N_{G_0}(M_0^{+ 1})}=\{ M_0^{+ 1}, M_0^{+ 1}\rtimes C_2\}\cup U,\mbox{ if $q_0+1+3m_0 | q+1$ and}$\\
$U_{N_{G_0}(M_0^{+ 1})}=\{ M_0^{+ 1}\}\cup U, \mbox{ otherwise,}$\\
where $U$ may contain only cyclic subgroups of order $2$ or $1$.
\end{prop}
\begin{proof}
Assume that $g\in G$ such that $G_0^{(g)}$ contains a nontrivial element
 $m\in M_0^{+1}$. Then by Lemma \ref{6tolrelpr} c), $M_0^{+1}\leq G_0^{(g)}$,
 thus $(M_0^{+1})^{g^{-1}}\leq G_0$.
 Using  the fact that the
Hall subgroups of order $q_0+1+3m_0$  are conjugate in $G_0$, we have that
there is an element $r\in  G_0$ such that
$ (M_0^{+ 1})^{g^{-1}r}=M_0^{+1}$ and thus $r^{-1}g\in N_G(M_0^{+ 1})$. Hence $g\in G_0N_G(M_0^{+1})$.\\

To prove  the second part, suppose that the element $g\in G$ has the property
 that $G_0^{(g)}\leqslant N_{G_0}(M_0^{+1})$. We have seen in Lemma \ref{ungpcomp}, that if $G_0^{(g)}$ contains a nontrivial  $3$-element,   then it
 contains the center of a Sylow $3$-subgroup of $G_0$, thus it also contains a subgroup isomorphic to
  $C_3^2$. This  contradicts  our assumption $G_0^{(g)}\leqslant N_{G_0}(M_0^{+1})$.  Hence in this case
 $G_0^{(g)}$ cannot contain nontrivial  $3$-elements.
 If $G_0^{(g)}$ does not contain nontrivial elements, whose order divides $q_0+1+3m_0$, this subgroup can   be isomorphic either to $C_2$ or to $\{1\}$. Assume that $G_0^{(g)}$ contains  nontrivial elements from $M_0^{+1}$. By the first part of the proof,
 we have that there is an element $r\in G_0$ such that $r^{-1}g\in N_G(M_0^{+1})$. Using Lemma \ref{abcd} b) with $H=N_G(M_0^{+1}),t=r$  and $T_0=G_0$ we have that\\
$G_0^{(g)}=(G_0\cap N_G(M_0^{+1}))^{(r^{-1}g)}=(N_{G_0}(M_0^{+1}))^{(r^{-1}g)},$
where $r^{-1}g\in N_G(M_0^{+1})\setminus G_0$.
Then $M_0^{+1}\leqslant N_{G_0}(M_0^{+1})^{(r^{-1}g)}$ by the choice of $r^{-1}g$.
Recall that $N_{G_0}(M_0^{+1})=M_0^{+1}\rtimes \langle t_1 \rangle$, where $t_1$ is an element of order $6$. Observe that both $N_{G_0}(M_0^{+1})$ and $N_{G_0}(M_0^{+1})^{r^{-1}g}$ are Frobenius groups with the same Frobenius kernel. Thus $N_{G_0}(M_0^{+1})^{(r^{-1}g)}\neq M_0^{+1}$, if and only if there is an element $r_1\in M_0^{+1}$ such that
\begin{gather}\tag{*}\label{*2}
\langle t_1\rangle^{r^{-1}gr_1}\cap \langle t_1\rangle\neq 1,
\end{gather}
where $r^{-1}gr_1\in N_{G}(M_0^{+1})\setminus N_{G_0}(M_0^{+1})$. Depending on the relation of $q_0$ and $q$, the subgroup $M_0^{+1}$ is contained in one of
  $\tilde M^{+ 1}$,  $\tilde M^{-1}$ or $\tilde M$.\\\indent
First assume that $M_0^{+1}\leqslant \tilde M^{+1}$.
Then $N_G(M_0^{+1})=\tilde M^{+1}\rtimes \langle t_1\rangle$ is also a Frobenius group with Frobenius complement $\langle t_1\rangle$. The equation
 (\ref{*2}) implies that $r^{-1}gr_1\in\langle t_1\rangle$, which is a contradiction. Thus $G^{(g)}=M_0^{+1}$ in this case. This can really occur, e.g. if we choose $g\in \tilde M^{+1}\setminus G_0$.


If $M_0^{+1}\leqslant \tilde M^{-1}$, then the proof is similar. \\\indent

Finally let us assume that $M_0^{+1}\leqslant \tilde M$.
 Thus $N_{G}(M_0^{+1})=(\tilde M\times V)\rtimes\langle t_1\rangle$.
 Let $m$ be a generator of $\tilde M$ and $r^{-1}gr_1=t_1^am^bv_1$, where
$a,b\in \mathbb{Z}$,  and $v_1\in V\setminus \{1\}$. Suppose that the $3$-element $t_1^2$ acts on $V=\{1,v_1,v_2,v_3\}$ as $v_1^{t_1^2}=v_2$, $v_2^{t_1^2}=v_3$ and $v_3^{t_1^2}=v_1$. The involution $t_1^3$ is centralizes $V$, thus $v_1^{t_1}=v_3$, $v_3^{t_1}=v_2$, $v_2^{t_1}=v_1$. 
We have that $t_1^{t_1^am^bv_1}=t_1^{m^bv_1}=(t_1[t_1,m^b])^{v_1}=t_1^{v_1}[t_1,m^b]=t_1[t_1,v_1][t_1,m^b]$.
By Theorem \ref{ree} (d), 
 $[t_1,m^b]\ne 1$ if and only if $\frac{q+1}{4}\not|b$.
Thus $\langle t_1 \rangle ^{t_1^am^bv_1}\cap \langle t_1 \rangle = \{1\}$, if
$\frac{q+1}{4}\not|b$.
However, $\langle t_1 \rangle^{t_1^av_1}=\langle t_1 \rangle ^{v_1}$ and since
$t_1^{v_1}=v_3t_1$, $(t_1^{v_1})^2=v_2t_1^2$, $(t_1^{v_1})^3=t_1^3$,
we have that $\langle t_1 \rangle^{t_1^av_1}\cap \langle t_1 \rangle=\langle t_1^3 \rangle$.

Thus if  $G_0^{(g)}$  contains nontrivial elements from $M_0^{+1}$
and $M_0^{+1}\leq \tilde M$ then $G^{(g)}$ can be either $M_0^{+1}$ or $M_0^{+1}\rtimes \langle t_1^3\rangle$. These cases in fact occur.
 If $g\in \tilde M\setminus M_0$, then $G_0^{(g)}=M_0^{+1}$ and if $g\in V\setminus \{1\}$, then $G_0^{(g)}=M_0^{+1}\rtimes \langle t_1^3\rangle$.
\end{proof}

\begin{prop}\label{ungm-} Using Notation \ref{Hjel} we have that
if  $G_0^{(g)}$ contains a nontrivial element from $M_0^{-1}$, then $g\in G_0  N_G(M_0^{-1})$. In particular\\
$U_{N_{G_0}(M_0^{- 1})}=\{ M_0^{ -1}, M_0^{- 1}\rtimes C_2\}\cup U,\mbox{ if $q_0+1-3m_0|q+1$ and}$\\
$U_{N_{G_0}(M_0^{- 1})}=\{ M_0^{- 1}\}\cup U, \mbox{ otherwise,}$\\
where $U$ may contain only cyclic subgroups of order $2$ or $1$.
\end{prop}
\begin{proof}
The proof is  similar to the previous one.
\end{proof}

\begin{prop}\label{ungm} Using Notation \ref{Hjel} we have that
if  $G_0^{(g)}$ contains a nontrivial element from $M_0$, then $g\in G_0  N_G(M_0)$. In particular\\
$U_{N_{G_0}(M_0)}=\{M_0\times V\}\cup U,$\\
where $U$ may contain only cyclic subgroups of order $2$ or $1$.
\end{prop}
\begin{proof}
Assume that $g\in G$ such that $G_0^{(g)}$ contains a nontrivial element
 $m\in M_0$.
Then by Lemma \ref{6tolrelpr} a) we have that $M_0 \leq G_0^{(g)}$, hence
$M_0,M_0^{g^{-1}}\leq G_0$.
 Since the Hall subgroups of order $\frac{q_0+1}{4}$  are conjugate in $G_0$,
 there exists an element $r\in G_0$ such that $M_0^{g^{-1}r}= M_0$ and thus
 $r^{-1}g\in N_G(M_0)$. Hence $g\in G_0N_G(M_0)$\\
To prove  the second part, suppose that the element $g\in G$ has the property that $G_0^{(g)}\leqslant N_{G_0}(M_0)$ holds.
 We have seen in Lemma \ref{ungp}, that if a  nontrivial $3$-element lies in $G_0^{(g)}$,
 then $G_0^{(g)}$ contains the center of a Sylow $3$-subgroup of $G_0$, and hence it also contains a subgroup isomorphic to
$C_3^2$. This contradicts our assumption $G_0^{(g)}\leq N_{G_0}(M_0)$. Thus $G_0^{(g)}$ cannot contain  nontrivial
$3$-elements.
 If $G_0^{(g)}$ does not contain elements, whose order divides $\frac{q_0+1}{4}$, then it   can  only be isomorphic  to $C_2^3$, $C_2^2$,  $C_2$ or $\{1\}$.
 By Lemma \ref{6tolrelpr} d) we know that $C_2^2$ and $C_2^3$ cannot happen.
 Assume that $G_0^{(g)}$ contains a nontrivial  element from $M_0$.
 By the first part of the proof, we have that there is an element $r\in G_0$
 such that $r^{-1}g\in N_G(M_0)$.
Using Lemma \ref{abcd} b) with $T_0=G_0$, $H=N_G(M_0)$  and $t=r$, we have that\\
$G_0^{(g)}=(G_0\cap N_G(M_0))^{(r^{-1}g)}=(N_{G_0}(M_0))^{(r^{-1}g)},$
where $r^{-1}g\in N_G(M_0)\setminus G_0$.
By Lemma \ref{6tolrelpr} a) we have that
 $M_0\times V\leqslant (N_{G_0}(M_0))^{(r^{-1}g)}$. We want to prove that
there is equality here.
 Recall that $N_{G_0}(M_0)=(M_0\times V)\rtimes \langle t \rangle$, where $t$ is an element of order $6$ and $N_G(M_0)=(M\times V)\rtimes \langle t\rangle$.\\
If $N_{G_0}(M_0)\cap N_{G_0}(M_0)^{r^{-1}g}> M_0\times V$  holds, then this intersection contains $\langle t^k \rangle ^x$ for some integer $0<k<6$ and $x\in M_0\times V$. Hence it also contains $\langle t^k\rangle$. But then
$\langle t^k \rangle\leq \langle t \rangle ^{r^{-1}gy}$, for some $y\in M_0\times V$.
 Since $r^{-1}g\in N_G(M_0)$, $r^{-1}gy=t^smv$ for some integer $s$,  $v\in V$ and $m\in M\setminus M_0$.
Thus $\langle t^k\rangle\leq \langle t \rangle^{mv}$ and hence
$\langle t^k\rangle= \langle t^k \rangle^{mv}$.
Then $[mv,t^k]\in (M\times V)\cap \langle t^k \rangle =1$.
Thus ${t^k}^{m}={t^k}^{v}$ and hence ${t^k}^{m^2}={t^k}$.
Since by Lemma  \ref{6tolrelpr} a),  $C_G(m^2)=M\times V$ and $0<k<6$, this cannot happen and we have
$G_0^{(g)}=N_{G_0}(M_0)^{(r^{-1}g)}= M_0\times V$.

This case in fact occurs.
Let $g\in M\setminus M_0$. Then 
 $G_0^{(g)}\ne G_0$ and  $M_0\times V\leqslant G_0^{(g)}$.
Thus, $G_0^{(g)}\leq N_{G_0}(M_0)$, or $G_0^{(g)}\leq C_{G_0}(i)$, for some involution $i\in G_0$. By the results above,
 in the first case $G_0^{(g)}=M_0\times V$.
 We will show that the second case alone does not occur. Let us suppose by contradiction that
 $M_0\times V\leq G_0^{(g)}\leq C_{G_0}(i)$. Since $g\in M\leq C_G(i)$, thus
 $i^{g}=i$.
By Lemma \ref{abcd} b) with $T_0=G_0$, $H=C_{G}(i)$  and $t=1$, we have that
$G_0^{(g)}=(G_0\cap C_G(i))^{(g)}= C_{G_{0}}(i)^{(g)}$.
 Let $C_{G_0}(i)=\langle i \rangle \times L$, where $L\simeq PSL(2,q_0)$.

 By a result of Levchuk and Nuzhin, see \cite[Lemmas 5 and 6]{zll},
  if ${G_0}^{(g)}$ is solvable then in our case
it is contained in $N_{G_0}(M_0)$, and we are done. If it is nonsolvable, then since it contains $M_0\times V$
it can only be $C_{G_0}(i)=\langle i \rangle \times L$ from the possible
list of subgroups.
Thus $C_{G_{0}}(i)^{(g)}=C_{G_0}(i)$. However, by Proposition \ref{ungpcomp}, we know that $G_0^{(g)}=C_G(i)^{(g)}$ must contain the center of a Sylow $3$-subgroup, which by Theorem \ref{ree} (b) cannot happen. 
\end{proof}

\begin{prop}\label{urq0}  Let $R_0$  be a maximal subgroup in $G_0$ isomorphic
to $R(q_1)$. Then $U_{R_0}$ may contain some cyclic subgroups of order $2$ or $1$.
\end{prop}
\begin{proof}Let $g\in G$ be such that $G_0^{(g)}\leqslant R_0$.
 If $G_0^{(g)}$ would contain nontrivial
 elements whose orders divide either $\frac{q_0+1}{4}$,
$\frac{q_0-1}{2}$ or $q_0\pm3m_0+1$, then by
 Lemma \ref{6tolrelpr} a), b) and c) it would contain
 a subgroup isomorphic to $C_2^2\times C_{\frac{q_0+1}{4}}$,
 $C_{q_0-1}$ or $C_{q_0\pm3m_0+1}$, respectively.
Hence $G_0^{(g)}$ could not be inside $R_0$.
 Thus in $G_0^{(g)}$  a  nontrivial element must have
 order divisible by $2$ or $3$.
 If $G_0^{(g)}$ contains a nontrivial   $3$-element, then by Lemma \ref{ungpcomp} a) $Z(P_0)\simeq C_3^{2n_0+1}\leqslant G_0^{(g)}$ and so $G_0^{(g)} $  cannot be
 inside  $R_0$.
 Furthermore, $G_0^{(g)}$ cannot contain more than one involution.
 Otherwise, if it contains two commuting involutions, then  by Lemma \ref{6tolrelpr} d) $C_2^2\times C_{\frac{q_0+1}{4}}\leqslant G_0^{(g)}$, which is a contradiction.
If $G_0^{(g)}$ contains two non-commuting involutions,
then they generate a dihedral group.
It either contains a Klein four subgroup, or an element of odd order.
None of them can happen.
  Thus $G_0^{(g)}$ can only be $\{1\}$ or a cyclic subgroup of order $2$.\end{proof}

Now we determine $U_{C_{G_0}(i)}$.

\begin{prop}\label{ucgi}
Let $i$ be an involution in $G_0$. Then\\
$U_{C_{G_0}(i)}=\{1, \langle k\rangle, W_0, M_0\times V\ | \langle k\rangle \simeq C_2,\ [i,k]=1,\ i\in W_0\simeq C_{q_0-1}, \ i\in V, M_0\times V\simeq C_{\frac{q_0+1}{4}}\times C_2^2\}.$\\
Morover, every  cyclic subgroup  of order $2$ and $q_0-1$ in $G_0$ occurs
as $G_0^{(x)}$ for some suitable $x\in G$.
\end{prop}

\begin{proof}
Let $g\in G$ be such that $G_0^{(g)}\leqslant C_{G_0}(i)$.
Let us suppose that $p\in G_0^{(g)}$ is an element of order $3$ and denote by $P_0$ the Sylow 3-subgroup of $G_0$, which contains it. Then by Lemma \ref{ungpcomp} we have that $Z(P_0)\leq G_0^{(g)}\leqslant C_{G_0}(i)$.  
However, by Theorem \ref{ree} (b), $Z(P_0)\cap C_{G_0}(i)=\{1\}$, which is a contradiction.\\
Hence, since $C_{G_0}(i)\simeq \langle i \rangle \times PSL(2,q_0)$,
  the only prime order elements in $G_0^{(g)}$ are  of order $2$
 and of orders that are
 divisors of $\frac{q_0-1}{2}$ or $\frac{q_0+1}{4}$.

 By Lemma \ref{6tolrelpr} a), b) and d) we know that $G_0^{(g)}$ is a cyclic subgroup of order $2$ or contains either a cyclic subgroup of order $q_0-1$ or a subgroup isomorphic to ${C_{\frac{q_0+1}{4}}\times C_2^2}$.

 By Theorem \ref{reszcsoplista}  about the list of subgroups of $PSL(2,q_0)$,  the subgroups of $C_{G_0}(i)$, which can contain $C_{q_0-1}$ or $C_{\frac{q_0+1}{4}}\times C_2^2$ and do not contain nontrivial $3$-elements
 are isomorphic to $C_{q_0-1}$, $D_{q_0-1}\times C_2$,  $C_{\frac{q_0+1}{4}}\times C_2^2$ or $(C_{\frac{q_0+1}{4}}\times C_2^2)\rtimes C_2$.

 On the other hand, if $G_0^{(g)}\leqslant N_{G_0}(M_0)$, then by Proposition \ref{ungm}, the subgroup $G_0^{(g)}$ can be only $1$, $C_2$ or $M_0\times V$.
Since $(C_{\frac{q_0+1}{4}}\times C_2^2)\rtimes C_2$ is  isomorphic to a subgroup
of $N_{G_0}(M_0)$, we may exclude it from being isomorphic to $G_0^{(g)}$.

Moreover, according to Lemma \ref{6tolrelpr} d) we can also exclude
 that $G_0^{(g)}\simeq D_{q_0-1}\times C_2$. Summarizing: $G_0^{(g)}$ is isomorphic to one of the following groups: $1, C_2, C_{q_0-1}, C_{\frac{q_0+1}{4}}\times C_2^2$.

We know, by Proposition \ref{ungm}, that the subgroups of $C_{G_0}(i)$
isomorphic to   $C_{\frac{q_0+1}{4}}\times C_2^2\simeq M_0\times V$
  occur as $G_0^{(g)}$ (we may suppose that $g\in M\setminus M_0$ and $i\in V$).
   Now we give a construction for the remaining three cases.\\

Denote by $L$ the factor group $C_G(i)/\langle i\rangle$. Let
$L_0$  be the image of $C_{G_0}(i)$
 under the natural homomorphism $^-:C_G(i)\rightarrow L$.  
By \cite[Lemma 2.16]{F2}, there exists an element
 $\overline{l}_1\in L$ such that $L_0^{(\overline{l}_1)}=\{ 1\}$, since
$2(2n_0+1)<2n+1$.
We will need that more than $\frac{1}{2}q_0^2(q_0-1)^2(q_0+1)^2$ such elements exist. It is  shown in the proof of \cite[Lemma 2.16]{F2} the number of them is at least
$A=\frac{q(q-1)(q+1)- qq_0^4-2qq_0^3-3qq_0^2+2qq_0+2q+2q_0^5+3q_0^4-2q_0^3-q_0^2}{2}.$
If we use that $q_0\leq q/9$, $q_0^2\leq q/3$,  $q_0^3\leq q$ and $(2qq_0+)2q_0^5+5q_0^4-2q_0^3-2q_0^2\geq 0$, then we have that
$A-\frac{1}{2}q_0^2(q_0-1)^2(q_0+1)^2\geq\frac{q^3-q- q^3/9-2q^2-3q^2/3+2q-q^2}{2}=\frac{8/9\, q^3 -4\,q^2+q}{2}.$
Since it is bigger than $0$, if $q\geq 27$, we get that there are more than $\frac{1}{2}q_0^2(q_0-1)^2(q_0+1)^2$ elements $\overline{l}_1\in L$ such that 
$L_0^{(\overline{l}_1)}=\{1\}$. Hence, by taking inverse images, we have that
 there are more than $q_0^2(q_0-1)^2(q_0+1)^2$ elements $l_1\in C_G(i)$ such that $C_{G_0}(i)^{(l_1)}=\langle i \rangle $.\\

We want to prove that there exists an element $\overline l_2\in  L$ such that $L_0^{(\overline l_2)}\simeq C_{\frac{q_0-1}{2}}$. Let $U_0\leq L_0$
be a subgroup isomorphic to $C_{\frac{q_0-1}{2}}$.
By \cite[Theorem 1.2 (ii)]{F2}  we have that $N_L(U_0)\simeq D_{q-1}$ and
$N_{L_0}(U_0)\simeq D_{q_0-1}$. Let us denote by $U$ the maximal cyclic subgroup of $N_L(U_0)$.  Since $L_0$ is selfnormalizing in $L$, by 
 the proof of \cite[Lemma 2.16]{F2}, thus if we take an element $\overline l_2$
from $U\setminus U_0$ then $U_0\leq L_0^{(\overline l_2)}\ne L_0$. 
 Hence, by
\cite[Lemma 2.15]{F2}, we have that $U_0=L_0^{(\overline l_2)}$.
Taking an inverse image $l_2$ of   $\overline l_2$, 
 we have that
  $(C_{G_0}(i))^{(l_2)}\simeq \langle i\rangle \times C_{\frac{q_0-1}{2}}$.
We show that if $G_0^{(l_1)}$ and $G_0^{(l_2)}$ are subgroups of $C_{G_0}(i)$, then they are isomorphic to $C_2$ and $C_{q_0-1}$, respectively.
 Applying  Lemma \ref{abcd} c) with 
$H=\langle i \rangle$, $r=1$ and $g=l_1$ or $g=l_2$, respectively,
 we have that
   $G_0^{(l_1)}\cap C_{G_0}( i )=G_0^{(l_1)}\cap N_{G}(\langle i \rangle)=N_{G_0}(\langle i \rangle )^{(l_1)}=C_{G_0}(i )^{(l_1)}=\langle i \rangle$ and $G_0^{(l_2)}\cap C_{G_0}(i)=C_{G_0}(i)^{(l_2)}\simeq C_{q_0-1}$.
 If $G_0^{(l_1)}$ and $G_0^{(l_2)}$ are  contained in $C_{G_0}(i)$, then
 $G_0^{(l_1)}=\langle i \rangle $ and $G_0^{(l_2)}\simeq C_{q_0-1}$, containing
 $i$. 
For another involution $k\in C_{G_0}(i)$ we may choose an element $x\in G_0$
such that $i^x=k$, then $\langle k\rangle=G_0^{(l_1x)}$. So we are done with the first statement of the Proposition in this case.
 Let us suppose that we do not know if 
$G_0^{(l_1)}$ and $G_0^{(l_2)}$ are inside $C_{G_0}(i)$. 
  Since both $G_0^{(l_1)}$ and $G_0^{(l_2)}$ contain $ i$, by Proposition \ref{ungp}, \ref{ungm+}, \ref{ungm-}, \ref{ungm} and \ref{urq0}, depending on which maximal subgroup of $G_0$ contains them, they are isomorphic to one of the following:  $C_2$, $C_{q_0-1}$, $P_0'\rtimes C_2$,  $P_0\rtimes C_2,\  P_0\rtimes C_{q_0-1}$, $C_2^2\times C_\frac{q_0+1}{4}$, $C_{q_0+3m_0+1}\rtimes C_2$, $C_{q_0-3m_0+1}\rtimes C_2$  or $G_0$. The intersection of one of these subgroups with the centralizer $C_{G_0}(i)$ of one of their involutions $i$, is isomorphic to the following: $C_2, C_{q_0-1}, C_ 3^{2n_0+1}\rtimes C_2,C_3^{2n_0+1}\rtimes C_2, C_3^{2n_0+1}\rtimes C_{q_0-1}, C_2^2\times C_{\frac{q_0+1}{4}}, C_2,C_2, C_{G_0}(i)$. Since this intersection must be $C_2$ or $C_{q_0-1}$,
 this implies that the only  possibilities are: ${G_0}^{(l_2)}\simeq C_{q_0-1}$, 
 $G_0^{(l_1)}=\langle i\rangle$  or $G_0^{(l_1)}=M_0^{\pm1}\rtimes \langle i\rangle$ for some $M_0^{\pm1}\in Hall_{q_0\pm3m+1}(G_0)$.
In Proposition \ref{ungm+}
 and \ref{ungm-} we have seen that if $G_0^{(g)}\simeq C_{q_0\pm3m_0+1}\rtimes C_2$, then $(M_0)^{\pm1}\leqslant \tilde M\in Hall_\frac{q+1}{4}(G)$ and $g\in G_0(V\setminus \{1\})$, where $C_2^2\simeq V\leqslant C_G(\tilde M)$.
 There are $\frac{|C_{G_0}(i)|}{|N_{C_{G_0}(i)}(M_0^{\pm1})|}=\frac{q_0(q_0-1)(q_0+1)}{6}$ subgroups in  $G_0$, being 
 isomorphic to $C_{q_0\pm3m_0+1}$ and  having the property, that the
 normalizer contains $i$.  Since $V\leqslant C_G(i)$, there are at most $2\frac{q_0(q_0-1)(q_0+1)}{6}|C_{G_0}(i)V\setminus{1}|=q_0^2(q_0-1)^2(q_0+1)^2$
 elements $g\in C_G(i)$ such that $G_0^{(g)}\simeq C_{q_0\pm3m_0+1}\rtimes C_2$.
 However, we have seen that there are more elements $l_1\in C_G(i)$ such that $C_{G_0}(i)^{(l_1)}=\langle i\rangle$.
This implies that there are elements $l_1$, $l_2\in C_G(i)$ such that  $G_0^{(l_1)}=\langle i\rangle$ and $G_0^{(l_2)}\simeq C_{q_0-1}$. 

 Since both the involutions and cyclic subgroups of order $q_0-1$ are conjugate in $G_0$, every cyclic subgroup of order $2$ and $q_0-1$ occurs as $G_0^{(x)}$ for some $x\in G$.\\

Finally, we will show that there exits an element $g\in G$ such that $G_0^{(g)}=1$. Looking at the order of $G_0$, by Lemma \ref{6tolrelpr} it is enough to show, that there is an element $g$ in $G$ such that $G_0^{(g)}$  contains
 neither an involtion $i\in G_0$, nor elements from $P_0$ and $M_0^{\pm1}$ for every $P_0\in Syl_3(G_0)$ and $M_0^{\pm 1}\in Hall_{q_0\pm3m_0+1}(G_0)$.
 We  know that if $G_0^{(g)}$ contains elements from $P_0$ or $M_0^{\pm1}$, then by Lemma \ref{ungpcomp},  Proposition \ref{ungm+} and Proposition \ref{ungm-},
 $g\in G_0 N_G(P)$, (where $P$ is the unique Sylow $p$-subgroup of $G$ containing $P_0$) or $g\in G_0 N_G(M_0^{\pm1})$, respectively. If the involution $i$ is in $G_0^{(g)}$, then using Lemma \ref{abcd} a) with $H_0=\langle i\rangle $ and $T_0=G_0$ we have that, $g\in G_0 C_G(i)$. Thus, we can
 give an upper bound to the cardinality\\
$|\cup_{P_0\in Syl_3(G_0)} G_0 N_G(P)\bigcup\cup_{M_0^{\pm1}\in Hall_{q_0\pm3m_0+1}(G_0)}G_0 N_G(M_0^{\pm1})\bigcup \cup_{i\in G_0, \ o(i)=2}G_0 C_G(i)|\leq $ 
\\$
\sum_{P_0\in Syl_3(G_0)} |G_0 ||N_G(P)|/|N_{G_0}(P)|+$ $\sum_{M_0^{\pm1}\in Hall_{q_0\pm3m0+1}(G_0)}|G_0|| N_G(M_0^{\pm1})|/|N_{G_0}(M_0^{\pm1})|+$
\\$\sum_{i\in G_0, \ o(i)=2}|G_0|| C_G(i)|/|C_{G_0}(i)|=$\\
 $|G_0 | ([G_0:N_{G_0}(P_0)]\frac{|N_G(P_0)|}{|N_{G_0}(P_0)|}|+$ $\sum_{i=\pm1}|Hall_{q_0+i 3m_0+1}(G_0)|\frac{|N_G(M_0^{i})|}{|N_{G_0}(M_0^{i})|}+
 [G_0:C_{G_0}(i)]\frac{|C_G(i)|}{|C_{G_0}(i)|} )$

 Furthermore, we use that $|Hall_{q_0\pm 3m_0+1}|(G_0)\leq \frac{q_0^3(q_0^2-1)(q_0+3m_0+1)}{6}$, $|N_G(M_0^{\pm1})|\leq |N_G(M^{+1})|=6(q+3m+1)$ and $|G_0 \cap N_G(M_0^{\pm1})|\geq |N_{G_0}(M_0^{-1})|=6(q_0-3m_0+1)$. Thus the cardinality of the above set is at most
$q_0^3(q_0-1)(q_0^3+1)\Big((q_0^3+1)\cdot\frac{q^3(q-1)}{q_0^3(q_0-1)}+2\frac{q_0^3(q_0^2-1)(q_0+3m_0+1)}{6}\cdot\frac{ 6(q+3m+1)}{6(q_0-3m_0+1)}+$$+q_0^2(q_0^2-q_0+1)\frac{ q (q-1)(q+1)}{q_0 (q_0-1)(q_0+1)}\Big)=$
$(q_0^3+1)^2q^3(q-1)+\frac{1}{3}q_0^6(q_0^2-1)^2(q_0+3m_0+1)^2(q+3m+1)+q_0^4(q_0^2-q_0+1)^2q(q-1)(q+1)$
A naive upper bound for this is:
$(2q_0^3)^2q^3(q-1)+\frac{1}{3}q_0^6(q_0^2)^2(3q_0)^2(3q)+q_0^4(q_0^2)^2q(q-1)(2q)=$
$(q-1)(4q_0^6q^3+\frac{1}{3}q_0^{12}\frac{q}{q-1}27+2q_0^8q^2).$
Using that $q>3$ (actually we have $q\geq 27$), we have that
$\frac{q}{q-1}\leq \frac{3}{2}$ and $q_0^3\leq q$, and so the cardinality of the above set is at most
$(q-1)(4q^5 + 27/2q^4+2q^5)\leq 20(q-1)q^5$. For $q\geq 27$
this is obviously smaller than $|G|=(q-1)q^3(q^3+1)$, and as we have seen before, this implies that there is an element $g\in G$ such that $G_0^{(g)}=\{1\}$.
\end{proof}

\begin{tetel}\label{depthr(q0)} We have that for maximal subgroups $G_0\simeq R(q_0)$ in $R(q)$, where $q_0>3$   $\mathcal{U}_1(G_0)=\mathcal{U}_2(G_0)$,  $d_c(G_0,G)=4$ and $d(G_0,G)=3$.
\end{tetel}
\begin{proof}
We use Notations \ref{3jel} and  \ref{Hjel}.
 By Propositions \ref{ungp}, \ref{ungm+}, \ref{ungm-} \ref{urq0} and \ref{ucgi} we know that
  $\mathcal{U}_{\;1}(G_0)=\bigcup_{P_0\in Syl_3(G_0)}U_{N_{G_0}(P_0)}\bigcup\{H | H\simeq 1, C_2, C_{q_0-1}, C_{q_0+1\pm3m_0}, C_2^2\times C_{\frac{q_0+1}{4}}\}$ and sometimes
 $C_{q_0+1\pm3m_0}\rtimes C_2$ also belong to it.
    Since $G_0^{(g_1,g_2)}=G_0^{(g_1)}\cap G_0^{(g_2)}$, the set $\mathcal {U}_{\;1}(G_0)$ is closed under intersection if and only if $\mathcal{U}_{\;1}(G_0)=\mathcal{U}_{\;2}(G_0)$.
 It is easy to see, that $\mathcal{U}_{\;1}(G_0)$  will be intersection closed, if $U_{N_{G_0}}(P_0)$ is intersection closed for every $P_0\in Syl_3(G_0)$. To show this, we display the subgroups  according to Proposition \ref{ungp} and \ref{ucgi} framing those subgroups, which are surely part of $\mathcal{U}_{\;1}(G_0)$.\\
{\footnotesize{ \small 
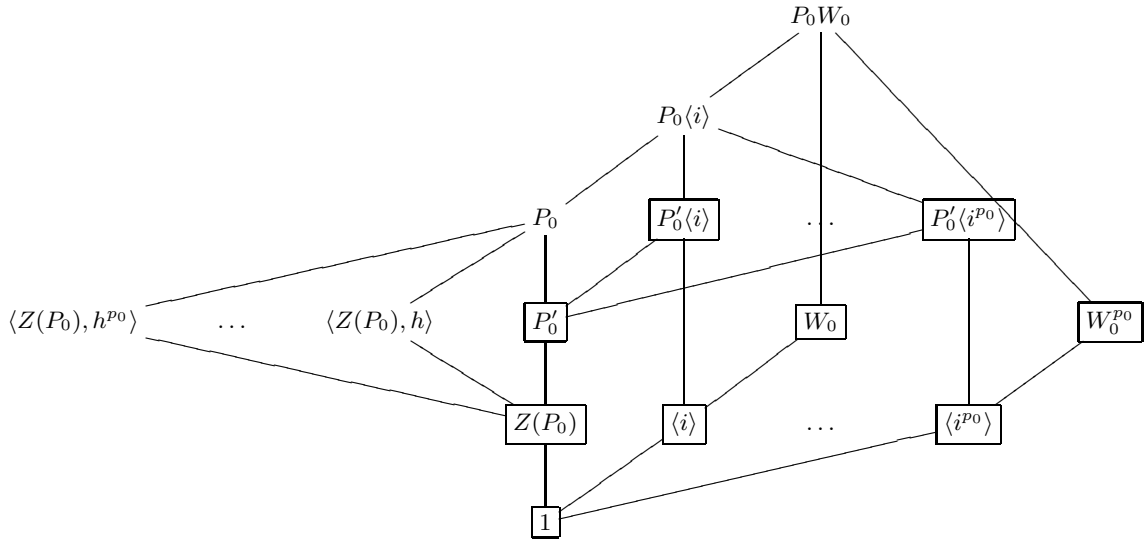
\xymatrix{
 & & & &   & P_0W_0      &\\
 & & && P_0\langle i\rangle\ar@{-}[ur]\\
 & & & P_0\ar@{-}[ur] & *+[F]{P_0'\langle i\rangle }\ar@{-}[u]& \dots & *+[F]{P_0'\langle i^{p_0}\rangle}\ar@{-}[ull]\\
 \langle Z(P_0),h^{p_0}\rangle\ar@{-}[urrr]&\dots &\langle Z(P_0),h\rangle\ar@{-}[ur] &*+[F]{ P_0'}\ar@{-}[u]\ar@{-}[ur]\ar@{-}[urrr]   &   & *+[F]{W_0}\ar@{-}[uuu]&      &*+[F]{W_0^{p_0}}\ar@{-}[uuull]\\
  & & & *+[F]{Z(P_0)}\ar@{-}[u]\ar@{-}[ul]\ar@{-}[ulll]& *+[F]{\langle i\rangle}\ar@{-}[uu]\ar@{-}[ur] & \dots & *+[F]{\langle i^{p_0}\rangle}\ar@{-}[uu]\ar@{-}[ur]\\
   & & &  *+[F]{1}\ar@{-}[u] \ar@{-}[ur]\ar@{-}[urrr]  &   &       &
}}} \\

It can be checked easily that any subgroups without frame can be skipped without any problem and the remaining  set of subgroups is still
 closed under intersection. Using  condition (1) in Definition \ref{jellemzes},  we have that $d_c(G_0,G)\leq 4$. Using condition (2) in Definition \ref{jellemzes},  we show that $d_c(G_0, G)>3$.

  Let $x_2\in G$ be such that $G_0^{(x_2)}=P_0'$ (actually we can choose any proper subgroup of $G_0$ from $\mathcal{U}_{\;1}(G_0)$ instead of $P_0'$) and let $x_1\in G_0\setminus N_{G_0}(P_0')$. Then $G_0^{(x_1,x_2)}=P_0'$. Assume by contradiction
 that there exists an element
 $y_1\in G$ such that $G_0^{(x_1,x_2)}=G_0^{(y_1)}$ and $h^{x_1}=h^{y_1}$ for all $h\in P_0'$.
Since $P_0'^{x_1}=P_0'^{y_1}$, we have that $P_0'^{x_1}\leqslant G_0^{y_1}$ and so $P_0'^{x_1}\leqslant G_0^{(y_1)}=P_0'$. By the choice of $x_1$, i. e. $x_1 \notin N_G(P_0')$,  we have a contradiction. Thus $d_c(G_0, G)>3$ and hence
$d_c(G_0, G)=4$.
Using Theorem \ref{inter}, we have that $d(G_0,G)=3$.
\end{proof}

\begin{rem} With similar methods we  also proved that if $G_0\simeq R(3)$
then $d_c(G_0,G)=4, d(G_0,G)=3$ also holds. Its proof is in the next section. 
\end{rem}

\section{The depth of $H\simeq R(3)$}

Let $H$ be a subgroup of $G\simeq R(3^{2n+1}), n>1$ that is isomorphic to $R(3)$. This subgroup is maximal in $G\simeq R(3^{2n+1})$, if $2n+1$ is prime. Let us list some properties of
 $H$. They can be easily checked e.g. in the GAP system, \cite{GAP}.

\begin{prop}\label{R3}
\begin{enumerate}[a)]
\item $|H|=3^3\cdot2^3\cdot 7=1512=(3^3+1)\cdot 3^3\cdot 2$.

\item The Sylow $3$-subgroups in $H$ form TI sets. 
Let $P_0\in Syl_3(H)$. Then $Z(P_0)=P_0'\simeq C_3$.
The normalizer of $P_0$ is $N_H(P_0)= P_0\rtimes \langle i\rangle$, where $i$ is an involution. We also have that $C_{P_0}(i)\simeq C_3$  and $C_{P_0}(i)\cap Z(P_0)=\{1\}$. Denote by $P_0^1$ the  elementary abelian subgroup generated by the elements of order $3$ in $P_0$. All elements of $P_0\setminus P_0^1$ are of order $9$ and their $3$-rd powers
belong to $Z(P_0)$.
$P_0$ has a reprezentation in triples, similar to $P$, $\{(x,y,z)|x,y,z\in GF(3)
\}$ and $\sigma $ acts identically on $GF(3)$. $P_0'=Z(P_0)=\{ (0,0,z)|z\in GF(3)\}$,
$P^1=P'\cap P_0=\{(0.y,z) | y,z\in GF(3) \}$.
\item Let $S\in Syl_2(H)$. Then $S$ is elementary abelian group of order $8$,
  $|N_H(S)|=8\cdot 7\cdot 3$ and $C_G(S)=S$. The $2$-subgroups of equal order are conjugate in $H$. The centralizer of an involution $i\in H$ is $C_H(i)\simeq\langle i \rangle \times  A_4$. It is contained in $N_H(S)$  for some $S\in Syl_2(H)$.

\item The Sylow $7$-subgroups in $H$ form TI sets. Let $M^{+1}_0\in Syl_7(G)$. Then the subgroup $N_H(M^{+1}_0)$ is a Frobenius group with kernel $M^{+1}_0$ and cyclic complement of order $6$.

\item Each maximal subgroup of $H$ is conjugate to one of the following:\\
$N_H(P_0)$,  $N_H(S)$,  $N_H(M^{+1}_0)$ or it equal to $H'\simeq PSL(2,8)$.\\
 Their orders are: $54$, $168$, $42$ and $504$, respectively.
\end{enumerate}
\end{prop}

We will prove the following Lemma for $R(3)$, which is similar to
 Lemma \ref{ungpcomp} and  Proposition \ref{ungm} for $R(q_0)$.

\begin{lemma}\label{unr3pcomp}Let $g\in G$ be an element
 such that $H^{(g)}\ne H$.  Let $P_0\in Syl_3(H)$ and $P_0\leqslant P\in Sly_3(G)$. Further let $N_H(P_0)=P_0\rtimes I$, where $I\simeq C_2$ and let $N_G(P)=P\rtimes W$ such that $I\leqslant W$. Then
\begin{enumerate}[(A)]
\item If $H^{(g)}$ contains a $3$-element $p\in P_0$,
then there exists an element $r\in H$ such that $r^{-1}g\in N_G(P)$.
 Moreover, this element $r$ can be chosen so that $r^{-1}g=p_1\in P\setminus P_0= (Z(P)\setminus Z(P_0))\cup (P'\setminus P^1_0Z(P))\cup (P\setminus  P_0P')$ and $H^{(g)}=(P_0I)^{(r^{-1}g)}=(P_0I)^{(p_1)}$. In particular, $Z(P_0)\leqslant H^{(g)}$
and $H^{(g)}\leq N_{H}(P_0)$.
    \begin{enumerate}[I.]
    \item{ If $r^{-1}g\in Z(P)\setminus Z(P_0)$ then $H^{(g)}=P_0$ or $P_0\langle i \rangle$.}
\item{ If $r^{-1}g\in P'\setminus P^1_0Z(P)$ then $H^{(g)}$
 is isomorphic to $P^1_0$ or $P^1_0\langle i \rangle$.
 Moreover, there exists an element $p_1  \in  P'\setminus P^1_0Z(P)$  such that $H^{(p_1)}=P^1_0$ and for every involution $j\in P_0I$ there exists
an element $p_2 \in  P'\setminus P^1_0Z(P)$ such that
   $H^{(p_2)}=P^1_0\langle j\rangle$.}
\item{ If $r^{-1}g\in P\setminus P_0P'$ then $H^{(g)}$ is $Z(P_0)$ or $ Z(P_0)\langle h\rangle$, where $h$ is an element of order $9$ in $P_0$.
 Moreover, there exists an element $p_1\in P\setminus P_0P' $ such that $H^{(p_1)}=Z(P_0)$.}

\end{enumerate}
\item If $H^{(g)}\leq N_{H}(P_0)$ and  $H^{(g)}$ does not contain any nontrivial $3$-elements from $P_0$, then $H^{(g)}$ is isomorphic  to $1$ or $C_2$.
\end{enumerate}
\end{lemma}
\begin{proof}
{\it (A)} Let $p\in H^{(g)}$ be a nontrivial $3$-element belonging to $P_0$.
 Let $S_0\in Syl_3(H)$ containing $p^{g^{-1}}$.
Then there exists an element $r\in H$ such that ${S_0}^r=P_0$. Since $p,p^{g^{-1}r}\in P_0\leq P$, and $P$ is $TI$, we have that
$r^{-1}g\in N_G(P)$. Thus  $g\in HN_G(P)\setminus H$.
 Then, since
 $HWP=HIP\cup H(W\setminus I)P= H(P\setminus P_0)\cup H\cup H(W\setminus I)P$,\\
 we have that
\[HN_G(P)\setminus H= H(P\setminus P_0) \cup H(W\setminus I)P.\]
\indent
We will show that if $g\in H(W\setminus I)P$ then $H^{(g)}$ does not contain
any non-trivial elements from $P_0$. Thus in part (A) $g\in H(P\setminus P_0)$
must hold.

Let $g=rw_1 p_1$, where $r\in H$, $w_1\in W\setminus I$ and $p_1\in P $.
We may suppose that $w_1\in W_{2'}$, otherwise, since $g=(ri)(iw_1)p_1$,
we choose $ri\in H$ instead of $r$ and $iw_1\in W\backslash I$ instead of $w_1$.
Using the fact that $N_{H}(P_0)\leq N_{H}(P)$ and
  Lemma \ref{abcd} c), with $G_0:=H$, $H:=P$ and $r:=r$,
 we have that $H^{(g)}\cap N_G(P)=
 N_{H}(P)^{(w_1p_1)}=N_{H}(P_0)^{(w_1p_1)}=(P_0I)^{(w_1p_1)}$.\\
Let $1\ne z\in Z(P_0)\cap (P_0I)^{w_1p_1}$. Then $z\in (P_0I)^{w_1}$
and $z^{w_1^{-1}}\in P_0\cap Z(P)=Z(P_0)$. Let $z_0=z^{w_1^{-1}}$. Then there exists an element $w_0\in I$ such that $z^{w_0}=z_0$. Hence $z^{w_0w_1}=z$.
Using Lemma \ref{hat} a) we have a contradiction.
 Hence $(P_0I)^{(w_1p_1)}$  does not
 contain  any nontrivial elements from $Z(P_0)$.\\
    Thus $P_0I^{(w_1p_1)}$  neither contains any elements of order $9$
 from $P_0$. We show by contradiction that $(P_0I)^{(w_1p_1)}$ does not contain any other $3$-elements.  Let $h$ be an element of order $3$ in
 $(P_0I)^{(w_1p_1)}$. Then $h \in P^1_0\backslash Z(P_0)$.
    Using the representation of the Sylow $3$-subgroups $P_0\in Syl_3(H)$ and $P\in Syl_3(G)$, we have that
  $h=(0,y_0,z_0)$ and $p_1^{-1}=(a,b,c)$, where
 $y_0\neq 0$, $y_0,z_0\in GF(3)$ and $a,b,c\in GF(q)$. Denote $w_1^{-1}$ by $w$. Then by Theorem  \ref{ree} h) and (by Lemma \ref{conj}) we have that
    \[h^{p_1^{-1}w}=(0,y_0,z_0)^{w}(0,0,2y_0a)^{w}\in P_0\cap P'=P^1_0.\]
 Denote the element  $(0,0,2y_0a)^w$ by $\zeta$. Since $(0,y_0,z_0)^w\zeta\in P^1_0\setminus Z(P_0)$, there exists an element $\zeta_0\in Z(P_0)$ such that $(0,y_0,z_0)^w\zeta=(0,y_0,z_0)\zeta_0$ or $(0,y_0,z_0)^w\zeta=(0,y_0,z_0)^{-1}\zeta_0$.
    Using Lemma \ref{hat} b) for $P^{'}$, we have in both cases a contradiction. In the first case $w\in W_{2'}$ fixes $(0,y_0,z_0)Z(P)$, and in the second case $w$ takes $(0,y_0,z_0)Z(P)$ to its inverse.
    Thus $H^{(g)}\cap P_0=\{1\}$,  which is a contradiction. Thus
$g\not\in H(W\backslash I)P$.
\\

 Hence if $H^{(g)}$ contains  a nontrivial
 $3$-element $p\in P_0$, then $g\in H(P\setminus P_0)$ i. e. $g=rp_1$ where $r\in H$ and $p_1\in P\setminus P_0$. Thus by Lemma \ref{abcd} d) with $H$ as $G_0$ and $Z(P_0)$ as $H_0$, we have
 that $H^{(g)}=(N_{H}(Z(P_0)))^{(p_1)}=(P_0I)^{(p_1)}$.\\
Obviously, $P\setminus P_0$ is a disjoint union of $Z(P)\setminus Z(P_0)$, $P'\setminus P_0^1Z(P)$ and $P\setminus P_0 P'$. Hence $p_1=r^{-1}g$ belongs to one of them.
\begin{enumerate}[I.]
\item{If $r^{-1}g=p_1\in Z(P)\setminus Z(P_0)$ then as $[P_0,p_1]=1$,
$H^{(g)}=(P_0I)^{(p_1)}\geqslant P_0$. Thus we have that $H^{(g)}$ is  $P_0 $ or $P_0I$.
}

\item{If $r^{-1}g=p_1\in P'\setminus P^1_0 Z(P)$, then $H^{(g)}=(P_0I)^{(p_1)}\geqslant P^1_0$, since $P'$ is elementary abelian.
If $H^{(g)}=(P_0I)^{(p_1)}$ would contain an element $h$
of order $9$, then $h^{{p_1}^{-1}}\in P_0$ would hold. Using the representation
of elements of the Sylow $3$-subgroup of $G$ in Theorem \ref{ree} h) and
Lemma \ref{conj}, we have that
if ${p_1}^{-1}=(0,b,c)$ and $h=(x_0,y_0,z_0)$ with $b\not\in GF(3)$,
 $x_0\ne 0$ and $x_0,y_0,z_0 \in GF(3)$, then $h^{{{p_1}}^{-1}}=(x_0,y_0,z_0-2x_0b)\in P_0$, which is a contradiction.
Thus $H^{(g)}$ can only be isomorphic
 to  $P^1_0$ or $P^1_0\rtimes C_2 $.}\\

Now we show, that for every involution $j\in P_0I$ if we take $p_1\in C_{P\setminus P_0}(j)$, which is nontrivial,  then $H^{(p_1)}=P^1_0\langle j\rangle$. By Theorem \ref{ree} b) $p_1\in P'\setminus Z(P)$ and $p_1\in P'\setminus P^1_0$. We show that $p_1\in P'\setminus P^1_0Z(P)$. For this consider the representation of the Sylow $3$-subgroup of $G$ in Theorem \ref{ree} h).
By Theorem \ref{ree} b) $C_P(j)\cap Z(P)=\{1\}$. We prove that
 for every $b\in GF(q)$ there is at most one $c\in GF(q)$ such that $(0,b,c)\in C_P(j)$.
Otherwise, if $(0,b,c_1),(0,b,c_2)\in C_P(j)$, then their quotient, $(0,0,c_1-c_2)\in C_P(j)\cap Z(P)=\{1\}$. Hence $c_1=c_2$.
 Since $|C_P(j)|=q$,   for every $b\in GF(q)$ there is exactly
 one $c\in GF(q)$  with $(0,b,c)\in C_P(j)$. Similar statement holds for
$C_{P_0}(j)$.  Hence if $b\in GF(3)$, then the
unique $c$ must be in $GF(3)$ and so $C_{P^1_0Z(P)}(j)=C_{P^1_0}(j)$.
 Accordingly, $p_1\in (P'\setminus P^1_0)\cap C_{P}(j)$ implies that
 $p_1\in P'\setminus P^1_0Z(P)$. Since $p_1\in N_G(P)$,  we have seen at the beginning of the proof
that
 $H^{(p_1)}=(P_0I)^{(p_1)}$. Using that $p_1\in P'\setminus P^1_0Z(P)$, we proved already that $(P_0I)^{(p_1)}$ is isomorphic to $P^1_0$ or
$P^1_0\langle i \rangle$. By an analogue of Lemma \ref{*} for $H$, we have that $j\in (P_0I)^{(p_1)}$.
Hence
$H^{(p_1)}=P^1_0\langle j\rangle$.\\

Finally we prove, that there is an element $p_1\in P'\setminus P^1_0 Z(P)$ such that $(P_0I)^{(p_1)}$ does not contain any involutions.
 Let $j\in (P_0I)^{(p_1)}$ be a fixed involution.
 Using Lemma \ref{*}  for $x=p_1\in P'\setminus P^1_0Z(P)$,
 by Theorem \ref{ree} (b) and (h) we have that
    \[|\{p_1\in P'\setminus P^1_0 Z(P)\ |\ j\in(P_0I)^{(p_1)} \}|= |(P'\setminus P^1_0 Z(P))\cap (P_0C_P(j))|\leq \]\[\leq|P^1_0C_P(j)\setminus P^1_0|=3(q-3).\]

    If for some $p_1\in P'\setminus P^1_0 Z(P)$, the subgroup
$(P_0I)^{(p_1)}$ contains an involution $j$, then
 $(P_0I)^{(p_1)}=P^1_0\langle j \rangle $. 
 Since the involutions are conjugate in $P_0I$ by elements of $P_0$,
 for every involution $j\in P_0I$ the same number of $p_1\in P'\setminus P^1_0Z(P)$ occurs such that $(P_0I)^{(p_1)}$ contains $j$.
Since by Lemma \ref{centralizerof3orderelement} we have that $C_P(p_1)=P'$,
the cosets of $P_0^1$ in $P_0$ move $p_1$ to $3$ different places. Thus,
from the number of elements $p_1$ belonging to  the involution $j$, one can get an upper bound on the
 the number of elements $p_1$ belonging
to any involution in $P_0I$, by multiplying with $3$.
    Since $P_0I$ contains $3^2$ involutions, there are $3$ subgroups
Thus we have that
    \[|\{p_1\in P'\setminus P^1_0 Z(P)\ |\ (P_0I)^{(p_1)}\mbox{ contains involutions} \}|\leq 9(q-3).\]
    Since $|P'\setminus P^1_0 Z(P) |=q(q-3)$ is bigger than
 $9 (q-3)$, there is an element $p_1\in P'\setminus P^1_0 Z(P)$
 such that $(P_0I)^{(p_1)}$ does not contain involutions.
Then, by the beginning of the proof of II.,  we have that $(P_0I)^{(p_1)} =P^1_0$  and hence
$H^{(p_1)}=(P_0I)^{(p_1)}=P^1_0$.

\item{If $r^{-1}g=p_1\in P\setminus {P_0P'}$ then $H^{(g)}=(P_0I)^{(p_1)}\geqslant Z(P_0)$. We show that $(P_0I)^{(p_1)}$ cannot contain noncentral elements of order $3$ in $P_0$.}
As before, we use the representation in  Theorem \ref{ree} (h) of $P$ and $P_0$.
Let us suppose that $h\in (P_0I)^{(p_1)}$ is a noncentral element of order $3$
in $P_0$. Let $p_1^{-1}=(a,b,c)$ and $h=(0,y_0,z_0)\in P_0$. Then
$a\not\in GF(3)$, otherwise $(a,b,c)=(a,0,0)(0,b,c+ab)\in P_0P'$, which is not the case. On the other hand,   $y_0\ne 0$, since otherwise  $h\in Z(P_0)$.
Then by Lemma \ref{conj}, $h^{{p_1}^{-1}}=(0,y_0,z_0+2y_0a)\in P_0$, which is a contradiction.\\

Now we prove that $H^{(g)}=(P_0I)^{(p_1)}$ does not contain any elements outside $P_0$.\\
Suppose this is not true and let $j$ be an involution in $(P_0I)^{(p_1)}$.
By Lemma \ref{*} we have  that $p_1\in P_0C_P(j)$ and so by Theorem \ref{ree} (b)
  we have that $p_1\in P_0P'$, which  contradicts  our assumption.\\

Now we show that if  there is an element $h\in P_0$ of order $9$
in $(P_0I)^{(p_1)}$ then $(P_0I)^{(p_1)}= Z(P_0)\langle h\rangle$.
It also follows as in the proof of Lemma \ref{ungpcomp}, but in this special case, it also trivially follows from the fact that $Z(P_0)\langle h\rangle=\langle h\rangle\simeq C_9$ is a maximal subgroup in $P_0$.\\

Thus we have seen that $p_1\in P\setminus P_0P'$ implies that
 $(P_0I)^{(p_1)}$ can be either $Z(P_0)\langle h\rangle$ for
 some element $h\in P_0$ of order $9$, or $(P_0I)^{(p_1)}=Z(P_0)$.

We show that there is an element $p_1\in P\setminus {P_0P'}$ such that $H^{(p_1)}=(P_0I)^{(p_1)}=Z(P_0)$.\\
Let $p_1^{-1}=(a,b,c)\in P\setminus P_0P'$.
Let $h=(x_0,y_0,z_0)\in P_0$ be an element such that $x_0\ne 0$, i.e. it is of order $9$, and let $h\in (P_0I)^{(p_1)}$. Then $h^{p_1^{-1}}\in P_0$ hence
$x_0,y_0,z_0$ satisfies

\begin{gather}\tag{I}\label{I}x_0(a\sigma )-a(x_0\sigma )\in GF(3)
\end{gather}
and
\begin{gather}\tag{II}\label{II}
-2x_0b+2y_0a-ax_0(x_0\sigma )+ax_0(a\sigma )\in GF(3)
\end{gather}

 Let $(x_0,y_0,z_0)\in P_0$ be fixed such that $x_0\neq 0$. Then
\[|\{p_1\in P\setminus P_0P' \ |\ (x_0,y_0,z_0)\in (P_0I)^{(p_1)}\}|\]\[\leq
|\{(a,b,c)\in GF(q)^3\ |\  a\notin GF(3), \mbox{ and b satisfies (II)}\}|\leq (q-3)3q.\]
If $p_1\in P\setminus P_0P'$ and $(P_0I)^{(p_1)}$ contains an element $h$ of order $9$, then $(P_0I)^{(p_1)}= Z(P_0)\langle h \rangle $, thus it
 contains exactly $6$ elements of order $9$.

So there are $\frac{3^3-3^2}{6}=3$  elements of order $9$ in $P_0$ giving
different subgroups in $P_0$ isomorphic to $ Z(P_0)\langle h \rangle$. Thus
\[|\{p_1\in P\setminus P_0P' \ |\ (P_0I)^{(p_1)}\mbox{ contains elements of order $9$}\}| \leq 9q(q-3).\]
Since $|P\setminus P_0P'|=q^2(q-3)$ is bigger than $9q (q-3)$, we have
 that there is an element $p_1$ such that $H^{(p_1)}=(P_0I)^{(p_1)}=Z(P_0).$

\end{enumerate}
\textit{(B)}
Since $H^{(g)}$ does not contain $3$-elements and $H^{(g)}\leqslant N_{H}(P_0)=P_0I$,  we have that $H^{(g)}$ is isomorphic  to $1$ or $C_2$.
\end{proof}
\begin{cor}
Using notation of Proposition \ref{R3} we have that
\[U_{N_{H}(P_0)}=\{ Z(P_0),P_ 0^1, P_0^1\langle i^{p_0}\rangle\ |p_0\in P_0\}\cup U\cup V,\]
where $V\subseteq \{Z(P_0)\langle h\rangle, P_0,P_0I\ |\ h\in P_0\setminus P_0^1, p_0\in P_0\}$ and $U\subseteq \{ 1, \langle i^{p_0}\rangle\ | \  p_0\in P_0\}$.
\end{cor}
\begin{prop}\label{unr3m+} Using Notation \ref{Hjel}  and  of Prop. \ref{R3},
 we have that
if  $H^{(g)}$ contains a nontrivial element  from $M_0^+$, then $g\in H  N_G(M_0^{+1})$. In particular
\[U_{N_{H}(M_0^{+ 1})}=\{ M_0^{+ 1}, M_0^{+ 1}\rtimes C_2\}\cup U,\mbox{ if $7| q+1$ and}\]
\[U_{N_{H}(M_0^{+ 1})}=\{ M_0^{+ 1}\}\cup U, \mbox{ otherwise,}\]
where $U$ may contain only cyclic subgroups of order $2$ or $1$.
In particular $H^{(g)}=M_0^{+1}\rtimes C_2$ holds for some $g\in G$ iff $7\mid q+1$ and $g\in H(V\setminus \{1\})$, where $V$ is the Klein subgroup in $N_G(M_0^{+1})$.
\end{prop}
\begin{proof}
Assume that $g\in G$ such that $H^{(g)}$ contains a nontrivial element
 $m\in M_0^{+1}$. Then  $M_0^{+1}\leqslant H^{(g)}$,
 thus $(M_0^{+1})^{g^{-1}}\leqslant H$.
 Using Sylow's theorem for $M_0^{+}\in Syl_7(H)$, 
 we have that
there is an element $r\in  H$ such that
$ (M_0^{+ 1})^{g^{-1}r}=M_0^{+1}$ and thus $r^{-1}g\in N_G(M_0^{+ 1})$. Hence $g\in HN_G(M_0^{+1})$.\\

To prove  the second part, suppose that the element $g\in G$ has the property
 that $H^{(g)}\leqslant N_{H}(M_0^{+1})$. We have seen in Lemma \ref{unr3pcomp}, that if $H^{(g)}$ contains a nontrivial  $3$-element, then it
 contains the center of a Sylow $3$-subgroup of $H$. However $N_{H}(M_0^{+1})$ contains only such $3$-elements, which centralize an involution.
  Thus by Proposition \ref{R3} b) we get a contradiction.  Hence, in this case
 $H^{(g)}$ cannot contain nontrivial  $3$-elements.
 If $H^{(g)}$ does not contain nontrivial elements, whose order divides $7$, this subgroup can   be isomorphic either to $C_2$ or to $\{1\}$. Assume that $H^{(g)}$ contains  nontrivial elements from $M_0^{+1}$.
 By the first part of the proof,
 we have that there is an element $r\in H$ such that $r^{-1}g\in N_G(M_0^{+1})$. Using Lemma \ref{abcd} b) with $T_0:=H$, $H:=N_G(M_0^{+1})$  and $t:=r$,
 we have that

\[H^{(g)}=(H\cap N_G(M_0^{+1}))^{(r^{-1}g)}=(N_{H}(M_0^{+1}))^{(r^{-1}g)},\]
where $r^{-1}g\in N_G(M_0^{+1})\setminus H$.
Then $M_0^{+1}\leqslant N_{H}(M_0^{+1})^{(r^{-1}g)}$ by the choice of $r^{-1}g$.
Recall that $N_{H}(M_0^{+1})=M_0^{+1}\rtimes \langle t_1 \rangle$, where $t_1$ is an element of order $6$. Observe that both $N_{H}(M_0^{+1})$ and $N_{H}(M_0^{+1})^{r^{-1}g}$ are Frobenius groups with the same Frobenius kernel. Thus $N_{H}(M_0^{+1})^{(r^{-1}g)}\neq M_0^{+1}$, iff there is an element $r_1\in M_0^{+1}$ such that
\begin{gather}\tag{*}\label{*3}
\langle t_1\rangle^{r^{-1}gr_1}\cap \langle t_1\rangle\neq 1,
\end{gather}
where $r^{-1}gr_1\in N_{G}(M_0^{+1})\setminus N_{H}(M_0^{+1})$. Depending on the relation of $3$ and $q$ the subgroup $M_0^{+1}$ is contained in one of
  $\tilde M^{+ 1}\in Hall_{q+3m+1}(G)$,  $\tilde M^{-1}\in Hall_{q-3m+1}(G)$ or $\tilde M\in Hall_{\frac{q+1}{4}}(G)$.\\\indent
First assume that $M_0^{+1}\leqslant \tilde M^{+1}$.
Then $N_G(M_0^{+1})=\tilde M^{+1}\rtimes \langle t_1\rangle$ is also a Frobenius group with Frobenius complement $\langle t_1\rangle$. The equation
 (\ref{*3}) implies that $r^{-1}gr_1\in\langle t_1\rangle$, which is a contradiction. Thus $H^{(g)}=M_0^{+1}$ in this case. This can really occur, e.g. if we choose $g\in \tilde M^{+1}\setminus H$.


If $M_0^{+1}\leqslant \tilde M^{-1}$, then the proof is similar. \\\indent

Finally let us assume that $M_0^{+1}\leqslant \tilde M$.
 Thus $N_{G}(M_0^{+1})=(\tilde M\times V)\rtimes\langle t_1\rangle$.
 Let $m$ be a generator of $\tilde M$ and $r^{-1}gr_1=t_1^am^bv_1$, where
$a,b\in \mathbb{Z}$,  and $v_1\in V\setminus \{1\}$. Suppose that the $3$-element $t_1^2$ acts on $V=\{1,v_1,v_2,v_3\}$ as $v_1^{t_1^2}=v_2$, $v_2^{t_1^2}=v_3$ and $v_3^{t_1^2}=v_1$. The involution $t_1^3$ is centralizes $V$, thus $v_1^{t_1}=v_3$ etc.
We have that $t_1^{t_1^am^bv_1}=t_1^{m^bv_1}=(t_1[t_1,m^b])^{v_1}=t_1^{v_1}[t_1,m^b]=t_1[t_1,v_1][t_1,m^b]$.
By Lemma \ref{centralOfqpm1}, $[t_1,m^b]\ne 1$ if and only if $\frac{q+1}{4}\not|b$.
Thus $\langle t_1 \rangle ^{t_1^am^bv_1}\cap \langle t_1 \rangle = \{1\}$, if
$\frac{q+1}{4}\not|b$.
However, $\langle t_1 \rangle^{t_1^av_1}=\langle t_1 \rangle ^{v_1}$ and since
$t_1^{v_1}=v_3t_1$, $(t_1^{v_1})^2=v_2t_1^2$, $(t_1^{v_1})^3=t_1^3$,
we have that $\langle t_1 \rangle^{t_1^av_1}\cap \langle t_1 \rangle=\langle t_1^3 \rangle$.

Thus if  $H^{(g)}$  contains nontrivial elements from $M_0^{+1}$
and $M_0^{+1}\leq \tilde M$ then $H^{(g)}$ can be either $M_0^{+1}$ or $M_0^{+1}\rtimes \langle t^3\rangle$. These cases in fact occur.
 If $g\in \tilde M\setminus M_0^{+1}$, then $H^{(g)}=M_0^{+1}$ and if $g\in V\setminus \{1\}$, then $H^{(g)}=M_0^{+1}\rtimes \langle t_1^3\rangle$.\\

\end{proof}

\begin{prop}\label{ree3u} 
If $H^{(g)}\neq H$, then it is contained in $ N_H(P_0)$ or $N_H(M^{+1}_0)$ for some $P_0\in Syl_3(H)$, $M^{+1}_0\in Syl_7(H)$ or it is isomorphic  a Klein  
four group.
\end{prop}
\begin{proof}
Let assume that $S\in Syl_2(H)$ is in $H^{(g)}$ for some $g\in G$. Since $S,S^{g^{-1}}$ are Sylow $2$-subgroups of $H$, there exists an element $r\in H$ such that $g^{-1}r\in N_G(S)$. However, $N_G(S)\leqslant H$ and so $g\in H$. Thus $H^{(g)}=H$.\\
Let assume that $M^{+1}_0\simeq C_7$ is contained in $H^{(g)}\neq H$ for some $g\in G$. Thus $H^{(g)}$ is contained in a  subgroup of $H$,  maximal with the property that it has some $7$-elements but not having $C_2^3$ as a subgroup. The maximal such subgroup in $N_H(S)$ or in $H'$ is isomorphic to $C_7\rtimes C_3$ or to $D_{14}$, respectively. Since both of them are contained in a conjugate of $N_{H}(M_0^{+1})$, we can suppose that $M^{+1}_0\leqslant H^{(g)}\leqslant N_H(M_0^{+1})$. \\

The largest subgroup of $H'$ containing $3$-elements  is isomorphic to $D_{18}$. This is a subgroup of $N_H(P_0)$ for suitable $P_0\in Syl_3(H)$. 
 If $H^{(g)}\neq H$ and it is not contained up to conjugacy in any of the subgroups 
 $N_H(P_0)$, $N_H(M_0^{+1})$ and $H'$, then  $H^{(g)}$ is a subgroup of $N_H(S)$ for some $S\in Syl_2(H)$,  not containing  $S$ and $7$-elements. 
 
Thus,  it is isomorphic to $A_4$ or a Klein four group. The subgroups of size
$12$ are all conjugate in $H$. 
However, since for an involution $i\in H$, the subgroup 
$C_H(i)\simeq \langle i \rangle\times A_4$, hence  by Lemma \ref{R3},
the involution of $A_4$
cannot be in the center of a Sylow $3$-subgroup of $H$.
So, by Lemma \ref{unr3pcomp}, $A_4$ cannot occur, either.
\end{proof}
\begin{prop}  Let $H$ be a maximal subgroup of $G$ isomorphic to $R(3)$.
   Then $\mathcal{U}_{\,1}(H)=\mathcal{U}_{\,2}(H)$. In particular $d_c(H,G)=4$ and $d(H,G)=3$.
\end{prop}
\begin{proof}
 From the previous proposition we get that $H^{(g)}$ can be isomorphic to one of the following: $H$, $C_7$ or $D_{14}$, $Z(P_0)\simeq C_3$, $P^1_0\simeq C_3^2$, $P^1_0\rtimes C_2\simeq C_3^2\rtimes C_2$, $C_9$ (two classes), $P_0$, $P_0\rtimes C_2$, $C_2$, $1$,  and $C_2^2$. Since $H^{(g_1,g_2)}=H^{(g_1)}\cap H^{(g_2)}$, to finish the proof it is enough to show that $\mathcal{U}_{\,1}$ is  intersection closed. However, we already have seen in Proposition \ref{ree3u} that there is an $g\in G$ such that $H^{(g)}$ is isomorphic to $Z(P_0)$, $P_0^1$, $P_0^1\rtimes C_2$.
The Klein subgroup  could only  occur in $\mathcal{U}_2(H)$ as the
  intersection with itself, it does not occur as the intersection of
 other subgroups in $\mathcal{U}_1(H)$.
If we show, that $1,C_2$ are in $\mathcal{U}_1(G)$, then we see that $\mathcal {U}_1(G)=\mathcal{U}_2(G)$.
 We have to show that there are elements $g_1$, $g_2\in G$ such that $H^{(g_1)}=1$ and $H^{(g_2)}\simeq C_2$. 
 Now we give constructions for these cases.\\

 Let $i$ be an involution and $C_G(i)=\langle i\rangle\times L$, where $L\simeq PSL_2(q)$ and let $C_H(i)=\langle i \rangle \times A$, where $A\simeq A_4$. By \cite[Lemma 2.7]{F} there is an element $l\in L$ such that $A^{(l)}=1$. Obviously $H^{(l)}$ contains $i$ and if $H^{(l)}\neq \langle i\rangle$, then $H^{(l)}\simeq M_0^{+1}\rtimes \langle i\rangle$ in the case $7\mid q+1$ , since $H^{(l)}\cap C_G(i)=\langle i\rangle .$
 However, by Proposition \ref{unr3m+} we know that $H^{(l)}$ can be $M_0^{+1}\rtimes \langle i\rangle$, if and only if  $l\in H (V\setminus \{1\})$ and $7| q+1$, where
$V$ is the Klein subgroup in $N_G(M_0^{+1})$. Since  $l\in L$ and we can suppose that
 $V\leqslant L$, we get that
 \[|\{l\in L\ |\ H^{(l)}=M_0^{+1}\rtimes \langle i\rangle\}|= |H(V\setminus \{1\})\cap L|=|A| 3=36.\]
 However, every involution is in $4$ conjugates of $N_H(M_0^{+1})$ and so
 \[|\{l\in L\ |\ H^{(l)}\simeq C_7\rtimes \langle i\rangle \mbox{ and contains $i$}\}|=36\cdot 4=144.\]
We will show there are more than $144$ elements in $L$ such that $A^{(l)}=1$.

 According to the proof of \cite[Lemma 2.7]{F}  we have that
 \[|\{A^l |\ l\in L\ \mbox{ and}\ A^{(l)}=1\}|\geq |L:N_L(A)| -|\{A^l\ne A |\
 l\in L \mbox{ and} \ A^{(l)} \mbox{ contains  involutions}\}|-\]
\[|\{A^l\ne A |\ l\in L \mbox{ and} \ A^{(l)} \mbox{ contains order $3$ elements}\}|-1\geq \]\
\[ \frac{q^3-q}{24}-3(\frac{q+1}{4}-1)-4(\frac{q}{3}-1)-1=\frac{q^3-q}{24}-
\frac{25q-63}{12},\] which is bigger than $144$ if $q\geq 27$. Hence $|\{l\in L\ |\ A^{(l)}=1\}|>144$.


Thus there is an element $l\in L$ such that $A^{(l)}=1$ and $H^{(l)}=\langle i\rangle$.\\

If $H^{(g)}$ does not contain  any $2$-elements,  $3$-elements or $7$-elements, then $H^{(g)}=1$. Let $i$ be an involution in $H$. Using the fact that $i\in H^{(g)}$ for some $g\in G$ iff $g\in HC_G(i)$ from Lemma \ref{unr3pcomp} and
Proposition \ref{unr3m+} we get that
    \[ \bigcup_{P_0\in Syl_3(H)}\{g\in G\ |P_0\cap H^{(g)}\neq1\}\cup \bigcup_{i\in H\ o(i)=2}\{g\in G\ | i\in H^{(g)}\}\cup \bigcup_{M_0^{+1}\in Syl_7(H)}\{g\in G\ | M_0^{+1}\leqslant H^{(g)}\}=\]
    \[ \bigcup_{P_0\in Syl_3(H), P_0\leqslant P\in Syl_3(G)}HN_G(P)\cup \bigcup_{i\in H\ o(i)=2}HC_G(i)\cup \bigcup_{M_0^{+1}\in Syl_7(H)}HN_G(M_0^{+1}).\]
    The size of this is at most
    \[28^2\cdot q^3(q-1) + 63^2\cdot (q^3-q)+36^2\cdot 6\cdot (q+3m+1).\]

    Since this is smaller that $|G|$, there is an element $g\in G$ such that $H^{(g)}=1$.\\

    Since $\mathcal{U}_{\,1}(H)=\mathcal{U}_{\,2}(H)$, the depth of $H$ in $G$ is $3$ or $4$. Using Defintion \ref{jellemzes} (ii) we show that $d_c(H, G)>3$.
  Let $x_2\in G$ be such that $H^{(x_2)}=P_0^1$ and let $x_1\in G_0\setminus N_{H}(P_0')$. Then $H^{(x_1,x_2)}=P_0'$. Assume by contradiction
 that there exists an element
 $y_1\in G$ such that $H^{(x_1,x_2)}=G_0^{(y_1)}$ and $h^{x_1}=h^{y_1}$ for all $h\in P_0^1$.
Since $(P_0^1)^{x_1}=(P_0^1)^{y_1}$, we have that $(P_0^1)^{x_1}\leqslant H^{y_1}$ and so $(P_0^1)^{x_1}\leqslant H^{(y_1)}=P_0^1$. By the choice of $x_1$, i. e. $x_1 \notin N_G(P_0^1)$,  we have a contradiction. Thus $d_c(H, G)>3$.\\


%


Using Theorem \ref{inter} we get that $d(H,G)=3$.

\end{proof}

\noindent
{\bf Acknowledgements:} The first and the second author were supported
by the National Research, Development and Innovation Office -NKFIH Grant No. 115288 and 115799.

\end{document}